\documentclass[10pt,reqno]{amsart}
\usepackage[margin=1in]{geometry} 
\usepackage[utf8]{inputenc}
\usepackage{fontenc}

\usepackage{amsthm,amsmath,amssymb}
\usepackage{enumitem,caption,graphicx}
\usepackage[colorlinks=true, pdfstartview=FitV, linkcolor=blue,citecolor=blue, urlcolor=blue]{hyperref}

\usepackage{array, babel, booktabs, cite, float, footmisc, kvoptions, multicol, nag, newtxmath, newtxtext, pdf14, pdftexcmds, ragged2e, upref, url, xcolor, xpatch, zref-base}

\numberwithin{equation}{section}

\newcommand{\arxiv}[1]{ArXiV preprint #1}

\def\tilde{\widetilde}
\renewcommand\hat{\widehat}
\def\ZZ{{\mathbb Z}}

\def\RR{{\mathbb R}}
\def\TT{{\mathbb T}}
\def\eps{\varepsilon}

\def\curl{\mathop\textrm{curl} \nolimits}
\def\ddiv{\mathop\textrm{div} \nolimits}

\theoremstyle{plain}
\newtheorem{theorem}{Theorem}[section]
\newtheorem{lemma}[theorem]{Lemma}
\newtheorem{conjecture}[theorem]{Conjecture}
\theoremstyle{definition}
\newtheorem{definition}[theorem]{Definition}

\newtheorem{remark}[theorem]{Remark}

\newcommand{\norm}[1]{\left\Vert#1\right\Vert}

\newcommand*{\Id}{\ensuremath{\mathrm{I}_d}}
\newcommand*{\Itwo}{\ensuremath{\mathrm{I}_2}}

\newcommand{\les}{\lesssim}

\begin{document}

\title[Anomalous diffusion via iterative quantitative homogenization]{Anomalous diffusion via iterative quantitative homogenization:\\ an overview of the main ideas}

\author[S.~Armstrong]{Scott Armstrong}
\address{Courant Institute of Mathematical Sciences, New York University, New York, NY 10012, USA.}
\email{\href{scotta@cims.nyu.edu}{scotta@cims.nyu.edu}}
\address{Laboratoire Jacques-Louis Lions, Sorbonne University, Paris, 75005, France.}
\email{\href{scott.armstrong@sorbonne-universite.fr}{scott.armstrong@sorbonne-universite.fr}}

\author[V.~Vicol]{Vlad Vicol}
\address{Courant Institute of Mathematical Sciences, New York University, New York, NY 10012, USA.}
\email{\href{vicol@cims.nyu.edu}{vicol@cims.nyu.edu}}

%

\begin{abstract}
Anomalous diffusion is the fundamental ansatz of phenomenological theories of passive scalar turbulence, and has been confirmed numerically and experimentally to an extraordinary extent. The purpose of this survey is to discuss our recent result, in which we construct a class of incompressible vector fields that have many of the properties observed in a fully turbulent velocity field, and for which the associated scalar advection-diffusion equation generically displays anomalous diffusion. Our main contribution is to propose an analytical framework in which to study anomalous diffusion via a backward cascade of renormalized eddy viscosities.
\hfill \today
\end{abstract}

\maketitle


\section{Introduction}

Let $u = u(t,x) \colon [0,\infty) \times \mathbb{T}^d \to \mathbb{R}^d$ be a given divergence-free velocity field, 
which is $\mathbb{T}^d$-periodic in space. Here $d\geq 2$ is the space dimension. We are interested in the setting where $u$ is continuous (in space and time), but has {\em low regularity}; for example, $u\in C^{\alpha}_{x,t}:= C^0_t C^\alpha_x \cap C^\alpha_t C^0_x$ for some $\alpha \in (0,1)$.  
The motivation to consider such vector fields stems from phenomenological theories of fully-developed hydrodynamic turbulence~\cite{Obukhov,Corrsin,Batchelor,Frisch,MK,shraiman,Warhaft,DSY}. The H\"older regularity exponent $\alpha = \frac 13$ is singled out by the Onsager theory of {\em ideal turbulence}~\cite{EyinkSreeni,EyinkSolo}, and this exponent plays a prominent role in this manuscript.

For such a given incompressible vector field $u \in C^\alpha_{x,t} $, we consider the advection-diffusion equation
\begin{equation}
\label{eq:drift:diffusion}
\partial_t \theta^\kappa + u \cdot \nabla \theta^\kappa 
= \kappa \Delta \theta^\kappa \,,
\qquad 
\theta^\kappa |_{t=0} = \theta_{\mathsf{in}} \,,
\end{equation}
which models the evolution of a {\em passive scalar} field $\theta^\kappa = \theta^\kappa(t,x) \colon [0,\infty) \times \mathbb{T}^d \to \mathbb{R}$ in a background fluid with velocity $u$. The parameter  $\kappa>0$ represents molecular diffusivity, or, in non-dimensional form, an inverse P\'eclet number. Since we are interested in the behavior of the fields $\theta^\kappa$ as $\kappa\to 0$, it is sufficient to consider $\kappa \in (0,1]$. 
The initial condition in \eqref{eq:drift:diffusion} is given by $\theta_{\mathsf{in}}= \theta_{\mathsf{in}}(x)$, which is taken to have zero-mean on $\TT^d$ (that is, $\fint_{\TT^d} \theta_{\mathsf{in}} dx = 0$), and is assumed to be sufficiently smooth (for instance, $\theta_{\mathsf{in}} \in H^1(\TT^d)$). The advection-diffusion equation~\eqref{eq:drift:diffusion} is to be solved for $t\geq0$; for simplicity, here we will only consider  $t\in[0,1]$. We note that if $u \in C^\alpha_{x,t}$ with $\alpha \in (0,1)$, then for any $\kappa>0$ the solution $\theta^\kappa$ of~\eqref{eq:drift:diffusion} is \emph{uniquely defined},\footnote{At lower regularity than the one considered in this paper, it is possible to construct non-unique weak solutions of the advection-diffusion equation~\eqref{eq:drift:diffusion}. This was first achieved in~\cite{MSz18,MSz19,cheskidov2021}, by treating the diffusion term as an error in an ``intermittent convex-integration scheme'' for the transport equation; this idea was introduced in~\cite{BV19,BCV22} for the 3D Navier-Stokes equations. In the works~\cite{MSz18,MSz19,cheskidov2021}, the velocity field $u$ and the scalar field $\theta^\kappa$ are constructed simultaneously, resulting in low regularity for both. See also the recent work~\cite{MS24} for a non-uniqueness proof based on stochastic
Lagrangians.} and experiences parabolic smoothing: \emph{in positive time}, we have $\theta^\kappa \in C^0_t C^{2,\alpha}_x \cap C^{1,\alpha/2}_t C^0_x$.  Also, note that the solution $\theta^\kappa$ continues to have zero-mean in positive time, namely $\fint_{\TT^d} \theta^\kappa(t,x) dx = 0$ for all $t\in (0,1]$.

\subsection{A rigorous definition of anomalous diffusion}
The fundamental {\em energy balance} for the advection-diffusion equation is derived by  multiplying~\eqref{eq:drift:diffusion} with $\theta^\kappa$, and using that $\ddiv u = 0$ to write
\begin{equation*}
\partial_t \bigl(\tfrac 12 |\theta^\kappa|^2\bigr)
+ \ddiv \bigl(u \; \tfrac 12 |\theta^\kappa|^2  \bigr) - \kappa  \Delta \bigl( \tfrac 12 |\theta^\kappa|^2 \bigr) 
+ \kappa |\nabla \theta^\kappa|^2
=0.
\end{equation*}
Integrating the above relation on $[0,t] \times \TT^d$ yields the energy balance\footnote{The quantity $\| \theta^\kappa\|_{L^2(\TT^d)}^2$  is the passive scalar's {\em variance} (the mean of $\theta^\kappa$ is fixed to equal zero), but we may also think about it as an {\em energy}.}
\begin{equation}
\tfrac 12 \norm{\theta^\kappa(t,\cdot)}_{L^2(\TT^d)}^2
+
\kappa \int_0^t \norm{\nabla \theta^\kappa(t^\prime,\cdot)}_{L^2(\TT^d)}^2  dt^\prime
=
\tfrac 12 \norm{\theta_{\mathsf{in}}}_{L^2(\TT^d)}^2
.
\label{eq:theta:balance}
\end{equation}
We emphasize that~\eqref{eq:theta:balance} may be justified rigorously (with $=$ sign)  under a very weak assumption on the incompressible velocity field $u$, namely $u \in L^\infty_t L^d_x$.

The uniform-in-$\kappa$ a priori estimate on $\theta^\kappa$ in $L^\infty_t L^2_x$ provided by \eqref{eq:theta:balance} shows that sequences $\{ \theta^{\kappa_j}\}_{j\geq 0}$, with $\kappa_j \to 0$ as $j\to \infty$, have sub-sequential weak-* limits in $L^\infty_t L^2_x$. Moreover, the linearity of \eqref{eq:drift:diffusion} implies that any such weak-* limit point $\bar \theta$ is a $L^\infty_t L^2_x$ weak solution of the transport equation
\begin{equation}
\label{eq:transport}
 \partial_t \bar \theta + \ddiv (u \bar \theta) = 0\,,
 \qquad 
 \bar \theta |_{t=0} = \theta_{\mathsf{in}} \, .
\end{equation}
Since  under weak convergence we have $\|\bar \theta(t,\cdot)\|_{L^2(\TT^d)} \leq \liminf_{\kappa\to0} \|\theta^\kappa(t,\cdot)\|_{L^2(\TT^d)}$ for $t\in [0,1]$, with~\eqref{eq:theta:balance} we obtain
\begin{align}
0 
\leq
\limsup_{\kappa \to 0} \kappa \int_0^t \norm{\nabla \theta^\kappa(t^\prime,\cdot)}_{L^2(\TT^d)}^2  dt^\prime
&= 
\tfrac 12 \norm{\theta_{\mathsf{in}}}_{L^2(\TT^d)}^2
- 
\liminf_{\kappa \to 0} \tfrac 12 \norm{\theta^\kappa(t,\cdot)}_{L^2(\TT^d)}^2
\notag\\
&\leq 
\tfrac 12 \norm{\theta_{\mathsf{in}}}_{L^2(\TT^d)}^2
- 
\tfrac 12 \Vert \bar \theta(t,\cdot)\Vert_{L^2(\TT^d)}^2
\label{eq:transport:dissipation}
\end{align}
for $t\in [0,1]$. The right side of \eqref{eq:transport:dissipation} depends only on the solution $\bar \theta$ of the transport equation~\eqref{eq:transport} which arose in the weak-* limit.\footnote{For vector fields $u$ whose regularity lies {\em below} $L^1_t W^{1,\infty}_x$,  e.g.~for $u\in L^1_t C^\alpha_x$ with $\alpha < 1$, the transport equations are expected to have non-unique solutions. Different subsequences $\kappa_j \to 0$ could in principle select different such solutions (see Theorem~\ref{thm:main} below). Therefore, \eqref{eq:transport:dissipation} should be written in terms of sequences $\kappa_j\to 0$, and for a {\em given} weak-* limit $\bar \theta$.} The bound on the energy dissipation rate obtained in \eqref{eq:transport:dissipation} shows that if {\em all $L^\infty_t L^2_x$ weak solutions} of the transport equation \eqref{eq:transport} have constant-in-time $L^2(\TT^d)$-norm, then the right side of \eqref{eq:transport:dissipation} vanishes identically, and thus  $ \limsup_{\kappa \to 0} \kappa \int_0^t \|\nabla \theta^\kappa(t^\prime,\cdot)\|_{L^2(\TT^d)}^2 dt^\prime  = 0$. 
Such a situation occurs for instance if the solution of \eqref{eq:transport} is uniquely obtained via ODE theory under the Cauchy-Lipschitz assumption $u \in L^1_t W^{1,\infty}_x$. At a lower space integrability of the velocity gradient, the Di Perna-Lions theory~\cite{diperna1989} guarantees that if $u \in L^1_t W^{1,1}_x$, then all $L^\infty_t L^2_x$ weak solutions of \eqref{eq:transport} are {\em renormalized} and thus they conserve their energy as a function of time.\footnote{See also the work of Ambrosio~\cite{Ambrosio04} for $u\in L^1_t \textrm{BV}_x$.}

The above discussion shows that if $\nabla u \in L^1_{x,t}$, then by the Di Perna-Lions theory we automatically have $\limsup_{\kappa \to 0} \kappa \int_0^t \| \nabla \theta^\kappa (t^\prime,\cdot)\|_{L^2(\TT^d)}^2 dt^\prime  = 0$ for $t\in [0,1]$. Turbulent flows $u$ however do not posses the property that $\nabla u \in L^1_{x,t}$; instead, we expect $u$ to have Onsager supercritical regularity, such as $u \in C^\alpha_{x,t}$ for some $\alpha\leq \frac 13$ (see~\cite{Frisch,EyinkSreeni,Buck2020,EyinkSolo} and the references therein). This leaves open the possibility that the right side of \eqref{eq:transport:dissipation} does not vanish identically, motivating the following definition:
\begin{definition}[Anomalous diffusion]
\label{def:ad}
Fix an incompressible vector field $u \in C^0_{x,t}$. 
If for any $\theta_{\mathsf{in}} \in H^1(\TT^d)$ with $\fint_{\TT^d} \theta_{\mathsf{in}}(x) dx=0$, there exists $0< \varrho \leq \frac 12$ such that the family of unique solutions $\{\theta^\kappa\}_{\kappa>0}$ of the Cauchy problem~\eqref{eq:drift:diffusion} satisfy
\begin{equation}
\limsup_{\kappa \to 0} \kappa \int_0^1 \!\!\! \int_{\TT^d} |\nabla \theta^\kappa(t^\prime,x) |^2 dx dt^\prime
\geq 
\varrho \|\theta_{\mathsf{in}}\|_{L^2(\TT^d)}^2
,
\label{eq:anomaly}
\end{equation}
we say the the advection-diffusion equation associated to $u$ exhibits {\em anomalous diffusion} on $[0,1]$; equivalently, {\em anomalous dissipation of the passive scalar's variance}.\footnote{For passive scalar transport, anomalous diffusion~\eqref{eq:anomaly} is equivalent to Lagrangian spontaneous stochasticity~\cite{DrivasEyink17}. Roughly speaking, this means that backwards-in-time transition probabilities of the SDE used to represent the advection-diffusion equation, do not converge as $\kappa\to 0$ to a Dirac mass around a deterministic Lagrangian trajectory.}
\end{definition}
 
The anomalous diffusion associated to the time interval $[0,1]$, as defined in Definition~\ref{def:ad}, may equivalently be introduced for any time interval $[0,t]$ with $t \in (0,1]$. Moreover, given any sequence $\kappa_n \to 0$, one may find a further subsequence $\kappa_{n_j}\to 0$ along which the {\em time dissipation measure} $\mathcal{E}(dt)$ is the weak-* limit in the sense of bounded measures
\begin{equation}
 \label{eq:dissip:measure}
\mathcal{E}(dt) := \lim_{j \to \infty} \kappa_{n_j} \norm{\nabla \theta^{\kappa_{n_j}} (t,\cdot)}_{L^2(\TT^d)}^2 
\,.
\end{equation}
Here we abuse notation, as clearly $\mathcal{E}$ also depends on the initial datum $\theta_{\mathsf{in}}$, and also on the chosen sub-sequence $\{\kappa_{n_j}\}_{j\geq 0}$. With \eqref{eq:transport:dissipation} and~\eqref{eq:dissip:measure}, relation~\eqref{eq:anomaly} becomes $\varrho \|\theta_{\mathsf{in}}\|_{L^2(\TT^d)}^2 \leq \int_0^1 \mathcal{E}(dt) \leq \frac 12 \|\theta_{\mathsf{in}}\|_{L^2(\TT^d)}^2$.

\subsection{Predictions from hydrodynamic turbulence}
\label{sec:OC}
It is widely accepted that if the vector field $u$ represents a homogenous isotropic turbulent velocity field, then solutions of~\eqref{eq:drift:diffusion}  exhibit anomalous diffusion in the sense of Definition~\ref{def:ad}. This prediction was first discussed by Obukhov~\cite{Obukhov}, who drew direct analogies to the anomalous dissipation of kinetic energy in a turbulent fluid, postulated by Kolmogorov's 1941 phenomenological theory~\cite{Kolmogorov41a}.  Compelling numerical and experimental evidence for anomalous diffusion is presented in~\cite{SreeniOld,shraiman,Warhaft,DSY}. In this setting, one expects that the time dissipation measure $\mathcal{E}(dt)$ is non-atomic, and is maybe even absolutely continuous with respect to Lebesgue measure on $[0,1]$.

Under the ansatz that the advection-diffusion equation~\eqref{eq:drift:diffusion} associated to $u$ exhibits anomalous diffusion (a so-called ``experimental fact''), and assuming that $u$ represents a homogenous isotropic turbulent velocity field exhibiting {\em monofractal scaling in the inertial range}, Obukhov~\cite{Obukhov} and Corrsin~\cite{Corrsin} have furthermore predicted that the sequence $\{\theta^\kappa\}_{\kappa>0}$ retains uniform-in-$\kappa$ H\"older regularity, with exponents which may be determined from scaling. 

In modern mathematical terms, in analogy with the Onsager conjecture for the incompressible Euler~\footnote{A weak solution $(u,p) \in L^2_{x,t} \times L^1_{x,t}$ of the incompressible Euler equations solves the system $\partial_t u + \ddiv (u\otimes u) + \nabla p=0$, with $\ddiv u = 0$, in the sense of space-time distributions.} equations~\cite{Eyink94,CET,EyinkSreeni,DLSzReview,DLSz13,Buck2020},\footnote{For the incompressible Euler equations, the Onsager dichotomy is a ``Theorem'', except for the most important endpoint case $\alpha = \frac 13$; see~\cite{CCFS,DuchonRobert} for the sharpest {\em rigidity} statements, and~\cite{Isett,BDLSZV,NovackVicol,GiriRadu,GiriKwonNovack} for the latest {\em flexibility} statements in the Onsager Theorem.} one may propose a {\em mathematical idealization} of the Obukhov-Corrsin prediction, and postulate the dichotomy:
\begin{enumerate}
\item[(i)] If~$u \in C^\alpha_{x,t}$ for some $\alpha \in (0,1)$ is divergence-free, and if the family of solutions~$\{\theta^\kappa\}_{\kappa>0}$ of \eqref{eq:drift:diffusion} are uniformly-in-$\kappa$ bounded in~$L^2_t C^{\bar \alpha}_x$ for some~$\bar \alpha > \frac{1-\alpha}{2}$, then there is no anomalous diffusion, namely $\lim_{\kappa\to 0} \kappa \|\nabla \theta^\kappa\|_{L^2_t L^2_x}^2 = 0$.
\item[(ii)] For each $\alpha \in (0,\frac 13)$,\footnote{The reason that we do not include here the full range $\alpha \in (0,1)$ is discussed in \S\ref{sec:open} below. } there exists divergence-free $u \in C^\alpha_{x,t}$, such that anomalous diffusion holds in the sense of Definition~\ref{def:ad}, with a non-atomic time-dissipation measure $\mathcal{E}(dt)$. Moreover, for~\emph{all} smooth initial conditions~$\theta_{\mathsf{in}}$, the solutions~$\{\theta^\kappa\}_{\kappa>0}$ of~\eqref{eq:drift:diffusion} are uniformly-in-$\kappa$ bounded in~$L^2_t C^{\bar\alpha}_x$ for~$\bar\alpha< \frac{1-\alpha}{2}$. 
\end{enumerate}
Part~(i) of the Obukhov-Corrsin  dichotomy is known, and follows directly from the commutator estimate of Constantin, E, and Titi~\cite{CET} (see e.g.~\cite{DEIJ} or~\cite{CCS22}).\footnote{Specifically, one may modify the mollification+commutator argument of~\cite{CET} to establish the bound $
\kappa \int_0^1 \norm{\nabla \theta^\kappa(\cdot,t)}_{L^2}^2 dt
\les
\kappa^{\frac{\alpha+2\bar \alpha-1}{\alpha+1}}
 (1 + \norm{u}_{L^\infty_t C^\alpha_x} )^{\frac{2(1-\bar \alpha)}{\alpha+1}}
\norm{\theta^\kappa}_{L^\infty_t C^{\bar \alpha}_x}^2
$.} As stated, part~(ii) of the dichotomy remains open (see \S\ref{sec:previous:results} for partial progress). The main purpose of our work~\cite{AV24} is to provide a new perspective on anomalous diffusion for the passive scalar equation, in the hope that these ideas will lead to a resolution of part (ii) of the Obukhov-Corrsin  dichotomy.

\subsection{Previous results} 
\label{sec:previous:results}
From a rigorous mathematical perspective, prior to our work~\cite{AV24}, we are aware of only two  (classes of) examples  of incompressible vector fields $u$ for which the associated passive scalar equation~\eqref{eq:drift:diffusion} exhibits (a form of) anomalous diffusion~\eqref{eq:anomaly}: the first is the stochastic Kraichnan model (introduced by Kraichnan in~\cite{K1}), while the second is a class of deterministic singularly-focusing mixing flows (introduced by Drivas, Elgindi, Iyer, and Jeong in~\cite{DEIJ}). We next discuss these examples, but do so {\em exceedingly briefly}. For a complete list of references and appropriate context, we refer the interested readers to the introductions of~\cite{AV24,DrivasEyink17,DEIJ,EyinkSolo,Leipzig,MFN}.
 
Kraichnan proposed in~\cite{K1} a model in which $u = u^\nu$ is a realization of a homogenous isotropic stationary zero-mean Gaussian random field, which is delta-correlated in time, and is colored in space in order to match the Kolmogorov 1941 phenomenological theory: the space covariance is fully-smooth at length scales $\lesssim \nu^{\frac 12}$ (a UV cutoff serving as a proxy for the Kolmogorov dissipation scale), and behaves like a $C^\alpha$ function at length scales $\gtrsim \nu^{\frac 12}$ (a synthetic power spectrum in the inertial range). For such a stationary Gaussian random field $u^\nu$, instead of~\eqref{eq:drift:diffusion} one studies the Stratonovich SDE $d \theta^{\kappa,\nu} + u^\nu \odot d \theta^{\kappa,\nu}  = \kappa \Delta \theta^{\kappa,\nu} dt$. 
The main result  is that in the joint limit $\nu,\kappa \to0$, and for a.e.~realization of the field $u^\nu$, and any $\theta_{\mathsf{in}} \in L^2(\TT^d)$ of zero-mean, anomalous diffusion holds; see~Bernard, Gawedski, and Kupiainen~\cite{Bernard}, Le Jan and Raimond~\cite{lejan2002}, or Falkovich, Gawedski, and Vergassola~\cite{falkovich2001}. A fully rigorous, complete, and concise proof of anomalous diffusion for the Kraichnan model (and its variants) may be found in the recent work of Rowan~\cite{Keefer}.  

The main drawback of the Kraichnan model stems from the white-noise temporal correlation of the vector field $u$, which is indeed so rough that it is likely entirely responsible for the anomalous diffusivity; see the mathematical discussion in~\cite{Keefer} and the numerical comparisons made in~\cite{sreenivasan2010}.

The first deterministic vector field $u$ for which the dissipation anomaly~\eqref{eq:anomaly} can be established rigorously, was constructed by Drivas, Elgindi, Iyer, and Jeong~\cite{DEIJ}. These authors show that for any $\alpha \in [0,1)$ and any dimension $d\geq 2$, there exists $u \in L^1([0,1];C^\alpha(\TT^d)) \cap L^\infty([0,1];L^\infty(\TT^d))$, such that $u(t,\cdot)$ is smooth for any $t<1$, and such that for any $\theta_{\mathsf{in}} \in H^2(\TT^d)$ which is sufficiently close (in $L^2$) to a an eigenfunction of the Laplacian, anomalous diffusion~\eqref{eq:anomaly} holds, with a purely-atomic time dissipation measure $\mathcal{E}(dt) = \rho(\theta_{\mathsf{in}}) \delta_{t=1}$. The construction in~\cite{DEIJ} is based on 
``quasi-self-similarly'' (as $t\to 1^-$)  accelerating an inviscid mixing dynamics~\cite{ACM1},  and treating the diffusion term as a perturbation. This is made possible by the singular nature of the ``speed-up'' of $u$ with respect to time, and the fact that at each time $t<1$ the field $u$ has only one ``active scale'' with respect to space. 

The ideas in~\cite{DEIJ} proved to be quite popular, leading to a number of further refinements~\cite{BrueDeLellis22,CCS22,JohanssonSorella2023,Cheskidov23,ElLi2023,JohanssonSorella2024,KeeferElias}. Notably, Colombo, Crippa, and Sorella~\cite{CCS22} gave examples of a smooth initial conditions $\theta_{\mathsf{in}}$, and of a divergence-free vector field $u \in C^0_t C^\alpha_x$ (for any $\alpha<1$), such that the associated family of solutions $\{ \theta^\kappa\}_{\kappa>0}$ is uniformly-in-$\kappa$ bounded in $L^2_t C^{\bar \alpha}_x$ for any $\bar \alpha < \frac{1-\alpha}{2}$, consistent with the Obukhov-Corrsin dichotomy; moreover, \eqref{eq:anomaly} holds with $\mathcal{E}(dt) = \rho(\theta_{\mathsf{in}}) \delta_{t=1}$, and sub-sequential limits  select different weak solutions of the transport equation.\footnote{At the lower regularity $u \in L^\infty_{x,t}$, the vanishing viscosity
limit fails to select not just unique solutions of the transport equation, but also physically admissible ones, in the sense of non-increasing energy/entropy.~\cite{Titi}}

This result was further refined by Elgindi and Liss~\cite{ElLi2023}, who for any $\alpha \in(0,1)$ provide a divergence-free vector field $u \in C^\infty([0,1]; C^\alpha(\TT^d))$ such that for {\em all} zero-mean and smooth initial conditions $\theta_{\mathsf{in}}$ (instead of just {\em one initial condition}) anomalous diffusion holds with $\mathcal{E}(dt) = \rho(\theta_{\mathsf{in}}) \delta_{t=1}$; moreover, $\{ \theta^\kappa\}_{\kappa>0}$ is uniformly-in-$\kappa$ bounded in $L^2_t C^{\bar \alpha}_x$ for any $\bar \alpha < \frac{1-\alpha}{2}$. We also mention the work of Johannson-Sorella~\cite{JohanssonSorella2024} who construct an {\em autonomous} divergence-free velocity field $u \in C^\alpha_{x}$ and a smooth zero-mean initial condition $\theta_{\mathsf{in}}$, such that~\eqref{eq:anomaly} holds; the associated time dissipation measure $\mathcal{E}(dt)$ is purely absolutely continuous  w.r.t. the Lebesgue on $[0,1]$. Lastly, we note the recent paper of Hess-Childs and Rowan~\cite{KeeferElias}  who for any $\alpha \in (0,1)$ construct $u \in C^0_t C^\alpha_x \cap C^{\frac{\alpha}{1-\alpha}}_t C^0_x$ such that for any zero-mean $\theta_{\mathsf{in}} \in L^1(\TT^2)$, the passive scalar $\theta^\kappa$ exhibits {\em asymptotic total dissipation} as $\kappa\to 0$, at $\{t=1\}$; that is,  $\lim_{\kappa \to 0} \|\theta^\kappa(1,\cdot)\|_{L^1(\TT^2)} = 0 $.

The main drawbacks of the constructions in~\cite{DEIJ,CCS22,ElLi2023} are that the anomalous dissipation occurs at only a discrete set of times ($\mathcal{E}(dt)$ is proportional to $\delta_{t=1}$), the vector field $u$ has  only one active scale at each time $t \in[0,1)$, and the advection-diffusion equation~\eqref{eq:drift:diffusion} is treated as a perturbation of the transport equation~\eqref{eq:transport}. Taken together, these properties are inconsistent with the observed properties of turbulent flows.

Note that the result in~\cite{ElLi2023} \emph{nearly solves} part (ii) of the Obukhov-Corrsin dichotomy discussed above, except for the singular nature of the time dissipation measure. In~\S\ref{sec:open} we argue that this seemingly minor caveat (the atomic nature of $\mathcal{E}(dt)$), combined with the fact that anomalous diffusion is meant to hold for all smooth zero-mean initial conditions $\theta_{\mathsf{in}}$, is in fact fundamental to the Obukhov-Corrsin dichotomy, and we present a conjecture in this direction (see Conjecture~\ref{eq:big:conj}).

\subsection{Main result}
Our goal in~\cite{AV24} was to build ``minimally realistic examples'' of ``fluid-like'' velocity fields $u$, for which anomalous diffusion holds in the sense of Definition~\ref{def:ad},\footnote{In particular,~\eqref{eq:anomaly} holds for {\em all} zero-mean smooth initial conditions.} and such that the time dissipation measure $\mathcal{E}(dt)$ defined in~\eqref{eq:dissip:measure} is {\em non-atomic}.\footnote{In particular, the anomalous dissipation occurs continuously in time.} 

From a broader perspective, our goal was to highlight what we believe to be the ``right reason'' for which anomalous diffusion holds: it is not that the underlying transport dynamics exhibits {\em mixing}, which is accelerated towards a discrete set of times; rather, we show that as one ``zooms out'' from small scales to large scales, one may observe a backwards cascade of renormalized ``eddy diffusivities'' taking shape, in such a way that the large scale motion experiences $O(1)$ diffusivity independently of the vanishing molecular diffusivity (as $\kappa \to 0$). Of course, this picture is not at all original and is indeed widely shared by physicists, where it goes under the name \emph{renormalization group}:~see for instance Frisch~\cite[Section 9.6]{Frisch}.
Our main result is:

\begin{theorem}[Theorem 1.1 in~\cite{AV24}]
\label{thm:main} 
Fix $d\geq 2$ and $\alpha \in (0,\frac 13)$.\footnote{The restriction $\alpha<\frac 13$ is intimately related to Conjecture~\ref{eq:big:conj} below.} There exists an incompressible vector field $u\in C^0([0,1]; C^\alpha(\TT^d))  \cap C^\alpha([0,1];C^0(\TT^d))$, such that for every $\theta_{\mathsf{in}} \in \dot{H}^1(\TT^d)$, the family of solutions $\{\theta^\kappa\}_{\kappa>0}$ of the advection-diffusion equation~\eqref{eq:drift:diffusion} with velocity $u$, exhibits anomalous diffusion in the sense of Definition~\ref{def:ad}, with a non-atomic time-dissipation measure. Moreover, we have:
\begin{itemize}[leftmargin=*]
\item The parameter $\varrho \in (0,\frac 12]$ appearing in~\eqref{eq:anomaly} depends only on $\alpha$, $d$, and a lower bound on the natural length scale of the initial datum, given by $L_{\mathsf{in}}:= \frac{\|\theta_{\mathsf{in}}\|_{L^2}}{\|\nabla \theta_{\mathsf{in}}\|_{L^2}}$. Specifically, for any $\eps>0$, we may find $c = c(d,\alpha,\eps)>0$ such that 
$\varrho = c  L_{\mathsf{in}}^{\frac{2(1+\alpha)}{1-\alpha}+\eps}$.
\item There exists a sequence $\kappa_j \to 0$, which depends on $\alpha$ and $d$ but is independent of $\theta_{\mathsf{in}}$, along which the lower bound in~\eqref{eq:anomaly} is realized; upon defining $\mathcal{K} = \cup_{j\geq 1} [\frac 12 \kappa_j, 2\kappa_j]$,  we have $\inf_{\kappa \in \mathcal{K}} \kappa \|\nabla \theta^\kappa\|_{L^2((0,1)\times\TT^d)}^2 \geq \varrho \|\theta_{\mathsf{in}}\|_{L^2(\TT^d)}^2$. 
\item There exists $\mu>0$, such that for any sequence $\kappa_i \to 0$ with $\{ \kappa_i\}_{i\geq 1} \subset \mathcal{K}$, the sequence of solutions $\{\theta^{\kappa_i}\}_{i\geq 1}$ is uniformly-in-$\kappa_i$ bounded in $C^{0,\mu}([0,1];L^2(\TT^d))$, as $\kappa_i \to 0$. In particular, the anomalous dissipation occurs continuously in time: along this subsequence we have that  $\int_0^t \mathcal{E}(dt^\prime)$ is H\"older continuous in $t \in [0,1]$. See~\cite[Remark~5.4]{AV24}.
\item There exists sequences $\{ \kappa_i\}_{i\geq 1}, \{\kappa_i^\prime\}_{i\geq 1} \subset \mathcal{K}$ such that as $\kappa_i,\kappa_i^\prime\to0$, the sequences $\{\theta^{\kappa_i}\}_{i\geq 1}$ and $\{\theta^{\kappa_i^\prime}\}_{i\geq 1}$ converge strongly in $C^0_t L^2_x$ to two {\em distinct} bounded weak solutions $\bar \theta$ and ${\bar\theta}^\prime$ of the transport equation~\eqref{eq:transport}. This demonstrates the lack of a selection principle for vanishing diffusivity limits of the advection-diffusion equation towards the transport equation. See~\cite[Proposition~5.5]{AV24}.
\item The constructed vector field $u$ is ``nearly'' a weak solution of the incompressible Euler equations, in the sense of~\cite[Proposition~2.8]{AV24}. 
\item The class of vector fields $u$ for which the all the conclusions of this Theorem hold is dense in the class of all smooth zero-mean divergence-free vector fields, with respect to the $C^0_{x,t}$ topology. See~Remark~\ref{rem:dense} below.
\end{itemize}
\end{theorem}

The construction of the vector field $u$ postulated in Theorem~\ref{thm:main} has many features in common with Nash-type convex-integration schemes for the incompressible Euler equations: multi-scale constructions with super-geometric scale separation, infinitely-many nested active scales with ``building blocks'' that are stationary Euler solutions, and the usage of Lagrangian flows to advect small-scale features by large-scale flows. On the other hand, the analysis of the passive scalars $\theta^\kappa$ is based on ideas from quantitative homogenization. 
Specifically, we need to perform homogenization  \emph{iteratively across many length scales}. In this way, we observe the emergence of a sequence of effective ``eddy'' diffusivities, which are recursively defined by homogenizing ``up the scales'', from the small scale at which molecular diffusion acts, to the large scale of the periodic box. In this process, every homogenization step resembles Taylor~\cite{Taylor} diffussion specialized to the case of a shear flow. An outline of the proof of Theorem~\ref{thm:main} is provided in~\S\ref{sec:proof:outline} below.

\subsection{Anomalous diffusion via weak solutions of 3D incompressible Euler}

We wish to highlight a very important extension of Theorem~\ref{thm:main}, which was recently obtained by Burczak, Sz\'ekelyhidi, and Wu in~\cite{Leipzig}. Whereas the vector field $u$ in Theorem~\ref{thm:main} may only be shown to be ``nearly'' a weak solution of the incompressible Euler equations on $\TT^d$,\footnote{See Proposition~2.8 in~\cite{AV24}.} the authors of~\cite{Leipzig} have shown that anomalous dissipation (as in Theorem~\ref{thm:main}) holds for large classes of vector fields $u \in C^\alpha_{x,t}$ (with $\alpha < \frac 13$), that are true weak solutions of three-dimensional incompressible Euler equations:

\begin{theorem}[Burczak, Sz\'ekelyhidi, Wu~\cite{Leipzig}]
\label{thm:Leipzig}
Fix $d=3$. For any $\alpha \in (0,\frac 13)$, there exists a weak solution $u\in C^0([0,1]; C^\alpha(\TT^3))  \cap C^\alpha([0,1];C^0(\TT^3))$ of the incompressible Euler equations, such that the advection-diffusion equation~\eqref{eq:drift:diffusion} with velocity $u$ exhibits anomalous diffusion, precisely as in Theorem~\ref{thm:main}. Additionally, the family of all such weak solutions $u \in C^{\alpha}_{x,t}$ of the Euler equations is dense in the class of all solenoidal vector fields, with respect to the $C^0_t H^{-1}_x$ topology.
\end{theorem}

The proof of Theorem~\ref{thm:Leipzig} uses the concepts and analytical methodology of our work~\cite{AV24}, and elegantly combines them with ideas from convex-integration / Nash-schemes for the incompressible Euler equations (the technical implementation in~\cite{BDLSZV} is of particular relevance). Notably, Burczak, Sz\'ekelyhidi, and Wu succeed in~\cite{Leipzig} to interlace the backwards energy cascade characteristic of convex integration schemes, with the backwards cascade of renormalized diffusivities which is underpinning iterated homogenization. At a technical level,~\cite{Leipzig} needs to contend with the fact that the correctors are not anymore explicit,\footnote{As opposed to~\cite{AV24},  the solution of the cell-problem in~\cite{Leipzig} is not explicit. Instead, the correctors are very well-approximated by explicit functions. This is known to be sufficient, because the space-homogenization step is stable under small perturbations in the diffusion matrix.} and with the fact that Isett's gluing technique~\cite{Isett} necessitates a more involved time-homogenization step.\footnote{Time-homogenization due to switching of flow-maps was  a difficulty already faced and resolved in~\cite{AV24}. The complication in~\cite{Leipzig} stems from the turning-on and turning-off of the fast oscillations in $u$, due to the so-called ``gluing step''.} Lastly, we note that the proof in~\cite{Leipzig} necessitates at least three space dimensions, due to the usage of non-intersecting Mikado flows~\cite{Daneri}. It is conceivable that a two-dimensional result is attainable by combining the ideas in~\cite{AV24,Leipzig} with the Newton-Nash scheme which was used to resolve the 2D Onsager conjecture~\cite{GiriRadu}.

\section{Background on quantitative homogenization}
\label{sec:hom}
Before turning to the proof of Theorem~\ref{thm:main}, it is instructive to briefly recall the basic two-scale expansion in parabolic periodic homogenization, emphasizing the relation between the effective equation and the phenomenon of \emph{advection-enhanced diffusion} (see \S\ref{sec:two-scale}), and then to briefly recall some of the basic themes of \emph{quantitative homogenization}, highlighted in the context of the homogenization problem of a periodic shear flow, and leading to the Taylor diffusion formula (see \S\ref{sec:dispersion}).

\subsection{The two-scale expansion in parabolic homogenization}
\label{sec:two-scale}

We begin with a general presentation of the parabolic two-scale expansion in the periodic setting; this material is standard (see for example~\cite{Jikov,GS,PS}). We consider a~$\ZZ \times \ZZ^d$-periodic, mean-zero, incompressible vector field~$b = b(t,x)$ and for $\eps,\kappa>0$, consider the advection-diffusion equation
\begin{equation}
\label{e.advec.u}
\partial_t \theta_\eps - \kappa \Delta \theta_\eps + \tfrac1\eps b(\tfrac{t}{\eps^2},\tfrac{x}{\eps}) \cdot \nabla \theta_\eps  = 0 
\,.
\end{equation}
By introducing a  stream matrix $\mathbf{s}$ for the vector field $b$ (that is, an anti-symmetric matrix such that~$-\ddiv \mathbf{s} = b$), the  advection term in~\eqref{e.advec.u} may be expressed as a second-order term $\frac{1}{\eps} b(\frac{t} {\eps^2},\frac{x}{\eps} ) \cdot \nabla \theta_\eps =  - \ddiv ( \mathbf{s}(\frac{t} {\eps^2},\frac{x}{\eps} )\nabla \theta_\eps  )$; in turn, this allows us to write~\eqref{e.advec.u} as
\begin{equation}
\label{e.advec.s}
\partial_t \theta_\eps - \ddiv \bigl( \mathbf{a}^\eps \nabla \theta_\eps \bigr)  = 0 \,,
\quad \mbox{where} \quad 
\mathbf{a}^\eps (t,x):=  \kappa \Id + \mathbf{s}( \tfrac{t} {\eps^2},\tfrac{x}{\eps}  ) = \mathbf{a}( \tfrac{t} {\eps^2},\tfrac{x}{\eps}  ) \,.
\end{equation}
Note that~\eqref{e.advec.s} contains a diffusion matrix $\mathbf{a}^\eps $ that oscillates fast (it depends on $\eps$), and is not symmetric (since $\mathbf{s}$ is skew-symmetric).  
Homogenization theory quantifies the convergence as $\eps \to 0$ (say with respect to the $L^2_{t,x}$ topology) of solutions $\theta_\eps$ of~\eqref{e.advec.s} to solutions $\theta$ of the \emph{effective equation}
\begin{equation}
\label{e.advec.s.hom}
\partial_t  {\theta} - \ddiv \bigl( \bar{\mathbf{a}} \, \nabla  {\theta} \bigr)  = 0 \,,
\end{equation}
where:
\begin{itemize}[leftmargin=*]
\item the homogenized/effective diffusion matrix $\bar{\mathbf{a}}$ is given by
\begin{equation}
\label{eq:abar:def}
\bar{\mathbf{a}} e = \bigl\langle\!\!\bigl\langle \mathbf{a} (e + \nabla \chi_e )
\bigl\rangle\!\!\bigl\rangle \,, \quad e \in\mathbb{S}^{d-1}\,, \quad \mathbf{a} :=  \kappa \Id + \mathbf{s} \,, 
\end{equation}
\item $\langle\hspace{-2.5pt}\langle\cdot \rangle\hspace{-2.5pt}\rangle$ denotes the space-time average of a~$\ZZ\times\ZZ^d$--periodic function, 
\item $\chi_e$ is the \emph{corrector} with slope~$e$; that is, the unique $\ZZ\times\ZZ^d$--periodic solution of the \emph{cell problem}
\begin{equation}
\label{eq:chie:def}
\partial_t \chi_e -\ddiv  \bigl( \mathbf{a} (e+\nabla \chi_e ) \bigr)= 0 \,,
\qquad
\bigl\langle\!\!\bigl\langle \chi_e \bigl\rangle\!\!\bigl\rangle = 0 \,.
\end{equation}
\end{itemize}
The above definitions show that the symmetric part of the effective diffusion matrix~$\bar{\mathbf{a}}$ is given by 
\begin{equation}
\label{eq:abar:sym}
\tfrac12 (\bar{\mathbf{a}} + \bar{\mathbf{a}}^t) _{ij} 
=
\kappa \delta_{ij} 
+ 
\kappa \, \bigl\langle\!\!\bigl\langle \nabla \chi_{e_i} \cdot \nabla \chi_{e_j}\bigl\rangle\!\!\bigl\rangle\,.
\end{equation}
The second term on the right side of \eqref{eq:abar:sym} is positive (as a nonnegative definite matrix), and therefore the symmetric part of~$\bar{\mathbf{a}}$ is larger than the original diffusion matrix~$\kappa \Id$. This effect is called the \emph{enhancement of diffusivity due to advection}.\footnote{The enhancement of diffusivity due to advection, from the point of view of homogenization theory, is by now classical, and we do not offer a complete summary of the relevant literature here. We instead refer the reader to~\cite{FP1,MK,PS} and the references therein.}

The convergence $\|\theta_\eps - \theta \|_{L^2_{t,x}} \to 0$ as $\eps \to 0$, where $\theta_\eps$ solves the microscopic diffusion equation~\eqref{e.advec.s} and $\theta$ is an appropriately chosen solution of the macroscopic or effective equation~\eqref{e.advec.s.hom} is summarized informally as: 
\begin{center} 
``the  operator $\partial_t - \ddiv(\mathbf{a}^\eps \nabla)$ \emph{homogenizes} as $\eps \to 0$ to the operator $\partial_t - \ddiv(\bar{\mathbf{a}} \nabla)$.''
\end{center}
A standard formalization of the above statement is achieved by proving that $\theta_\eps$ is well-approximated by its ``two-scale expansion'', which is defined in~\eqref{e.twoscale.wep.para} below.

The {\em parabolic} two-scale expansion is constructed as follows. In addition to the corrector with slope $e$, $\chi_e$ (defined in~\eqref{eq:chie:def}), and to the homogenized matrix $\bar{\mathbf{a}}$ (defined in~\eqref{eq:abar:def}), we introduce:
\begin{itemize}[leftmargin=*]
\item  the space-homogenized matrix $\bar{\mathbf{a}}(t)$ is defined for $t\in \RR$ by 
\begin{equation*}
\bar{\mathbf{a}}(t) e = \langle \mathbf{a}(t,\cdot) ( e + \nabla \chi_e(t,\cdot) ) \rangle,
\qquad 
e\in\mathbb{S}^{d-1}\,,
\end{equation*} where $\langle \cdot \rangle$ denotes the space average of a $\ZZ^d$--periodic function;
\item this allows us to define the matrix
\begin{equation*}
\mathbf{k}(t):= \int_0^t \bigl(\bar{\mathbf{a}}(t^\prime) - \bar{\mathbf{a}} \bigr) dt^\prime \,,
\end{equation*}
which by construction has zero-mean in time, i.e.~$\int_0^1 \mathbf{k}(t) dt = 0$;
\item integrating~\eqref{eq:abar:def} with respect to space alone implies that $\langle \partial_t \chi_e \rangle = 0$, and thus for each $t\in \RR$ we may define the \emph{time corrector} $h_e(t,\cdot)$ as the the $\ZZ^d$--periodic $H^1(\RR^d)$ solution of 
\begin{equation*}
-\Delta h_e(t,\cdot) = \partial_t \chi_e(t,\cdot)\,, \qquad \langle  h_e(t,\cdot) \rangle = 0 \,;
\end{equation*}
\item the purpose of introducing the space-homogenized matrix $\bar{\mathbf{a}}(t)$ and the time corrector $h_e(t,\cdot)$ is to ensure that for each $t\in \RR$ the vector field
\begin{equation*}
g_e(t,\cdot) := \mathbf{a}(t,\cdot) \bigl(e + \nabla \chi_e(t,\cdot) \bigr) - \bar{\mathbf{a}}(t) e + \nabla h_e(t,\cdot)
\end{equation*}
satisfies $\ddiv g_e(t,\cdot) = 0$ and $\langle g_e(t,\cdot) \rangle = 0$; in turn, this allows us to write $g_e(t,\cdot)$ as the divergence of a skew-symmetric matrix $\mathbf{m}_e(t,\cdot)$, which is the zero-mean solution of 
\begin{equation*}
-\Delta \mathbf{m}_{e,ij}(t,\cdot) = \partial_i g_{e,j}(t,\cdot)  - \partial_j g_{e,i}(t,\cdot) \,.
\end{equation*}
\end{itemize}
With this notation, the parabolic two-scale expansion is given by
\begin{equation}
\label{e.twoscale.wep.para}
 \tilde{\theta}_\eps(t,x):=
 \theta(t,x) + \eps \sum_{i=1}^d \partial_i \theta(t,x) \chi_{e_i}\bigl(\tfrac{t}{\eps^2},\tfrac{x}{\eps} \bigr) + \eps^2 \sum_{i,j=1}^d \mathbf{k}_{ij}\bigl(\tfrac{t}{\eps^2}\bigr) \partial_{ij} \theta(t,x)
 \,,
\end{equation}
where $\theta$ is a solution of~\eqref{e.advec.s.hom}. The function $\theta$ is determined only once we specify a domain and a parabolic boundary condition.

In order to show that $\tilde{\theta}_\eps$ is a good approximation of the solution~$\theta_\eps$ of~\eqref{e.advec.s} (again, by this we mean that one fixes a domain, and lets $\theta_\eps$ match the values of $\theta$ on the corresponding parabolic boundary), we simply plug in the ansatz~\eqref{e.twoscale.wep.para} into~\eqref{e.advec.s} and try to argue that the error we make is small. Using that~$\theta$ solves~\eqref{e.advec.s.hom} we arrive at:
\begin{equation}
\label{e.advec.s.approx}
\partial_t   \tilde{\theta}_\eps - \ddiv \bigl( \mathbf{a}^\eps \nabla  \tilde{\theta}_\eps \bigr) = f_\eps \,,
\end{equation}
where the ``error'' $f_\eps$ is given explicitly as
\begin{align}
\label{e.advec.s.f}
f_\eps(t,x) 
&:=  \eps 
\ddiv \left( 
\sum_{i=1}^d 
\left( h_{e_i} \Id - \chi_{e_i} \mathbf{a}  - \mathbf{m}_{e_i}  \right)\bigl(\tfrac{t}{\eps^2},\tfrac{x}{\eps} \bigr)
\nabla \partial_{x_i} \theta(t,x) \right)
\notag\\
&\quad 
+ \eps \sum_{i=1}^d 
\left( \chi_{e_i}\bigl(\tfrac{t}{\eps^2},\tfrac{x}{\eps} \bigr) \partial_{x_i}\partial_t \theta(t,x) 
+
h_{e_i} \bigl(\tfrac{t}{\eps^2},\tfrac{x}{\eps} \bigr)
\partial_{x_i}\Delta \theta (t,x)
\right)
+ \eps^2 \sum_{i,j=1}^d \mathbf{k}_{ij}\bigl(\tfrac{t}{\eps^2}\bigr) \partial_{ij} \partial_t \theta (t,x).
\end{align}
The proof of the statement ``the  operator $\partial_t - \ddiv(\mathbf{a}^\eps \nabla)$ \emph{homogenizes} as $\eps \to 0$ to the operator $\partial_t - \ddiv(\bar{\mathbf{a}} \nabla)$'', quantified by the convergence $\|\theta_\eps - \theta\|_{L^2_{t,x}} \to 0$ as $\eps \to 0$, now follows from~\eqref{e.twoscale.wep.para}, \eqref{e.advec.s.approx}, and~\eqref{e.advec.s.f}:
\begin{itemize}[leftmargin=*]
\item (good) bounds on the solution $\chi_e$ of the cell problem~\eqref{eq:chie:def} imply (good) bounds for the time correctors $h_e$, the skew-symmetric matrix $\mathbf{m}_e$, and the matrix $\mathbf{k}$; 
\item coupled with (good) derivative estimates for the macroscopic function $\theta$, which solves the homogenized problem~\eqref{e.advec.s.hom}, the aforementioned bounds show that the forcing $f_\eps$ appearing on the right side of \eqref{e.advec.s.approx} is $O(\eps)$ (say in $L^2_t H^{-1}_x$);
\item additionally, note that by the linearity of~\eqref{e.advec.s} we have that 
\begin{equation}
\label{eq:its:a:linear:world}
\partial_t(\tilde {\theta}_\eps - \theta_\eps) - \ddiv \bigl( \mathbf{a}^\eps \nabla (\tilde{\theta}_\eps - \theta_\eps) \bigr) = f_\eps, 
\end{equation}
and so by the standard parabolic energy estimate, $\tilde {\theta}_\eps - \theta_\eps$ is $O(\eps)$ (say in $L^2_{t,x}$); 

\item Since $\theta_\eps$ match the values of $\theta$ on the parabolic boundary of the spacetime domain, and~$\tilde {\theta}_\eps$ is an~$O(\eps)$ perturbation from~$\theta$, the difference~$\tilde {\theta}_\eps - \theta_\eps$ is at most~$O(\eps)$ on the boundary (say in~$L^\infty_{t,x}$); 

\item to conclude, note that from~\eqref{e.twoscale.wep.para} and the above mentioned bounds we directly obtain that $\tilde{\theta}_\eps - \theta$ is $O(\eps)$ (say in $L^2_{t,x}$) by standard energy estimates. 

\end{itemize}
The above items sweep under the rug the fact that the aforementioned bounds depend in a crucial way on the ellipticity constant of the matrix $\mathbf{a} = \kappa \Id + \mathbf{s}$, on the ellipticity ratio of $\mathbf{a}$, which quantifies the size of the skew part $\mathbf{s}$ to the size of the symmetric part $\kappa \Id$. They also depend on the size of the derivatives of the macroscopic function~$\theta$, since these appear in the definition of~$f_\eps$ in~\eqref{e.advec.s.f}. These considerations are crucial for the quantitative homogenization estimates which we discuss next.

\subsection{Quantitative homogenization for shear flows: Taylor diffusion}
\label{sec:dispersion}
Stated as in~\S\ref{sec:two-scale}, it appears that homogenization relies on \emph{asymptotic scale separation} (the limit $\eps\to 0$), which is a feature inconsistent with turbulence models.\footnote{See e.g.~the criticisms  in~\cite[page 225]{Frisch} and~\cite[page 304]{MK}.} Asymptotic scale separation is also inconsistent with the construction of the vector field $u$ in~\S\ref{sec:construct}. There, $\eps$ represents the ratio of two consecutive length scales, which for a fixed parameter $m$ in the iteration corresponds to $0< \frac{\eps_{m+1}}{\eps_m} < 1$. Note however that for $m=100$, say, the ratio $\frac{\eps_{101}}{\eps_{100}}$ is fixed, and cannot be sent to $0$. This is why \emph{quantitative  homogenization} methods need, and should, be used! Rather than arguing that some quantities converge to others in the limit of infinite scale separation, one needs to precisely quantify the length scales and time scales on which homogenization may be observed to have occurred, modulo small errors (which then need to be summed up across scales). Fortunately, quantitative homogenization rates for periodic coefficient fields are very explicit and quite well-understood (see~\cite{Jikov,ShenBook} for instance).

We will next highlight this quantitative homogenization perspective for the operator $\partial_t - \kappa\Delta + u \cdot \nabla$, when $u$ is a {\em two-dimensional shear flow}. For parameters $a,\eps>0$, let us consider  
\begin{equation}
\label{eq:u:a:eps:def}
    u = u_{a,\eps} = 2\pi a \eps \cos \big( \tfrac{2\pi x_2}{\eps} \big)  e_1 \,,
\end{equation}
so that $0 < \eps \ll 1$ is the scale of spatial oscillation (note that $u$ is time independent), and $a \gg 1$ is proportional to the Lipschitz norm of $u$ (note that $\|\nabla u\|_{L^\infty(\TT^2)} = 4 \pi^2 a $). Since we are two space dimensions, $\langle u \rangle = 0$, and $\ddiv u = 0$, we may write
\begin{equation}
\label{eq:psi:a:eps:def}
u_{a,\eps} = \nabla^\perp \psi_{a,\eps}\,,
\qquad \mbox{where} \qquad 
\psi_{a,\eps} = a \eps^2 \sin \big( \tfrac{2\pi x_2}{\eps} \big)\,.
\end{equation}
Associated to the stream function $\psi_{a,\eps}$ we may associate the anti-symmetric matrix $ \psi_{a,\eps} \Itwo^\perp $, where $\Itwo^\perp$ is the (rotation) matrix with components $(\Itwo^\perp)_{11} = (\Itwo^\perp)_{22} = 0$ and $(\Itwo^\perp)_{21} = - (\Itwo^\perp)_{12} = 1$. 
Then, in analogy to \eqref{e.advec.s}, we may rewrite 
\begin{equation}
\label{eq:bf:a:eps:shear:def}
\partial_t - \kappa\Delta + u_{a,\eps} \cdot \nabla
=
\partial_t - \ddiv \bigl( \mathbf{a}^\eps \nabla \bigr)
\end{equation}
where for compatibility of notation we have defined 
\begin{equation}
\mathbf{a}^\eps(x) := \kappa \Itwo + \psi_{a,\eps}(x) \Itwo^\perp = \mathbf{a}\bigl(\tfrac{x}{\eps}\bigr)
\,, 
\quad \mbox{with} \quad 
\mathbf{a}(y) := \kappa \Itwo + a \eps^2 \sin(2\pi y_2)\Itwo^\perp
\label{eq:bf:a:shear:def}
\,.
\end{equation}
Here $y$ is a placeholder for the fast space variable $\frac{x}{\eps}$.
In the example provided by~\eqref{eq:bf:a:eps:shear:def}--\eqref{eq:bf:a:shear:def} we may take advantage of the fact that $\mathbf{a}$ is independent of time and of the first space coordinate to deduce that the correctors $\chi_e$ (which solve the cell problem~\eqref{eq:chie:def}) are also independent of time and of the first space coordinates. In fact, we may quickly verify that the solution of~\eqref{eq:chie:def} is given by
\begin{equation}
\chi_{e_1}(y) = \tfrac{a\eps^2}{ 2\pi \kappa} \cos(2\pi y_2) \,,
\quad \mbox{and} \quad 
\chi_{e_2}(y) = 0 \,,
\label{eq:chi:e:shear:def}
\end{equation}
which is clearly $\ZZ\times\ZZ^2$--periodic. Plugging~\eqref{eq:chi:e:shear:def} into~\eqref{eq:abar:def} we obtain
\begin{equation}
\label{eq:bf:a:bar:shear}
\bar{\mathbf{a}}=  
\kappa \Itwo +
\begin{pmatrix}
        \tfrac{a^2 \eps^4}{2\kappa} && 0 \\
        0 && 0
\end{pmatrix}
 = \begin{pmatrix}
       \kappa + \tfrac 12 \kappa \cdot \bigl( \tfrac{a  \eps^2}{ \kappa} \bigr)^2 && 0 \\
        0 && \kappa
\end{pmatrix}
\,.
\end{equation}
The formula for the effective diffusion matrix $\bar{\mathbf{a}}$ given in~\eqref{eq:bf:a:bar:shear} is called the {\em Taylor diffusion formula}~\cite{Taylor} (specialized to the case of the shear flow in~\eqref{eq:u:a:eps:def}); see also the work of Aris~\cite{Aris}.

The Taylor diffusion formula for $\bar{\mathbf{a}}$ obtained in~\eqref{eq:bf:a:bar:shear} is notable because $\bar{\mathbf{a}}$ is diagonal (note that $\mathbf{a}$ had a non-constant anti-symmetric part), by definition it is constant in space and in time, and most notably, the enhancement of diffusion only occurs in the $e_1 \otimes e_1$ entry (a remnant of the fact that the $u$ in~\eqref{eq:u:a:eps:def} points in the $e_1$ direction but is independent of $x_1$). We also note that the quantity $\frac{a\eps^2}{\kappa}$ which appears squared as a factor for the strength of the enhancement is dimensionless: $a$ is the Lipschitz norm of a velocity field so it scales as $(\mbox{time})^{-1}$, the parameter $\eps$ has units of $(\mbox{length})$, while $\kappa$ is a molecular diffusivity so it has units of $(\mbox{length})^2 \cdot (\mbox{time})^{-1}$.

At this point we have simply derived an explicit formula for the homogenized matrix $\bar{\mathbf{a}}$, but we have not said anything \emph{quantitative} about the statement ``the  operator $\partial_t - \ddiv(\mathbf{a}^\eps \nabla)$ \emph{homogenizes} as $\eps \to 0$ to the operator $\partial_t - \ddiv(\bar{\mathbf{a}} \nabla)$.'' The following Lemma provides such a ``sample'' quantitative result; we do not claim that the following result is the sharpest possible, it is not the most general either; we present it here for pedagogical value, as it provides useful intuition about the quantitative bounds that may be extracted from the parabolic two-scale expansion.

\begin{lemma}[Quantitative homogenization and Taylor diffusion]
\label{lem:Taylor:quantitative}
Let $\mathbf{a}^\eps$ be defined by~\eqref{eq:bf:a:shear:def}, and let $\bar{\mathbf{a}}$ be defined by~\eqref{eq:bf:a:bar:shear}. The operator $\partial_t - \ddiv(\mathbf{a}^\eps \nabla)$ homogenizes to the operator $\partial_t - \ddiv(\bar{\mathbf{a}} \nabla)$, on 
\begin{equation}
\label{eq:scaling:stitches}
\mbox{\emph{time-scales} $\gtrsim a^{-1}\cdot  \bigl(1 + \frac{a\eps^2}{\kappa}\bigr) $ \qquad and  \qquad 
\emph{space-scales} $\gtrsim \eps \cdot  \bigl(1 + \frac{a\eps^2}{\kappa}\bigr)$ 
}\,.
\end{equation}
More precisely:
\begin{itemize}[leftmargin=*]
\item Define $\theta$ to be the $\TT^2$-periodic solution of the Cauchy problem for $\partial_t \theta - \ddiv(\bar{\mathbf{a}} \nabla \theta ) = 0$ (cf.~\eqref{e.advec.s.hom}), with smooth initial datum $\theta|_{t=0} = \theta_{\mathsf{in}}$. 
\item For $R>0$ define the \emph{parabolic cylinders adapted to $\bar{\mathbf{a}}$}, as
$Q_R = Q_R(t_0,x_0) = (t_0,x_0) + (-\frac{R^2}{\kappa},0]\times E_R(0)$, where $E_R$ is the ellipsoid defined by $\{x\colon \kappa^{\frac 12} |\bar{\mathbf{a}}^{-\frac 12} x|\leq R\} = \kappa^{-\frac 12} \bar{\mathbf{a}}^{\frac 12} B_R(0)$. In our case, due to~\eqref{eq:bf:a:bar:shear}, this ellipsoid has axes of length $\sim  R (1 + \frac{a\eps^2}{\kappa})$ in the $e_1$ direction, and $\sim R$ in the $e_2$ direction.
\item For every such parabolic cylinder $Q_{R}$ with $Q_{2R} \subseteq [0,1]\times \TT^2$, let $\theta_\eps$ be the solution of $\partial_t \theta_\eps - \ddiv(\mathbf{a}^\eps \nabla \theta_\eps) = 0$ in $Q_{R} $, with (parabolic) boundary data $\theta_\eps \equiv \theta$ on $\partial_{\mathrm{par}}Q_R$. \item Then, we have that
\begin{align}
\| \theta_\eps -   \theta\|_{L^\infty_t  \underline{L}^2_x (Q_R)} 
\les  
(1 + \tfrac{a\eps^2}{\kappa} ) 
\tfrac{\eps}{R}  \| \theta \|_{\underline{L}^2_{x,t}(Q_{2R})} 
\,.
\label{eq:wamboozle}
\end{align}
\end{itemize}
Here, $L^\infty_t \underline{L}^2_x(Q_R)$ and $\underline{L}^2_{x,t}(Q_R)$ denote  the usual volume-normalized norm (e.g.~integrals $\int_{Q_R}$ are replaced by $\fint_{Q_R}$).
In particular, when $R \gg \eps    (1 + \tfrac{a\eps^2}{\kappa} )$, we see that $\| \theta_\eps -   \theta\|_{L^\infty_t  \underline{L}^2_x (Q_R)}  \ll  \| \theta \|_{\underline{L}^2_{x,t}(Q_{2R})} $, which matches the heuristic claim regarding space-scales in~\eqref{eq:scaling:stitches}.
\end{lemma}

\begin{proof}[Brief proof of Lemma~\ref{lem:Taylor:quantitative}]
We wish to appeal to the parabolic two-scale expansion described in~\S\ref{sec:two-scale}. We have already computed the correctors with slope $e$, namely $\chi_e$, in~\eqref{eq:chi:e:shear:def}. Since $\mathbf{a} = \mathbf{a}(y_2)$ and $\chi_e = \chi_e(y_2)$ are both independent of $t$, the space homogenized matrix $\bar{\mathbf{a}}(t)$ is time independent, and is equal to $\bar{\mathbf{a}}$ for all $t\in \RR$. Hence, the matrix $\mathbf{k}(t) = 0$ for all $t\in \RR$, and the time correctors also vanish identically, $h_e(t,\cdot) = 0$ for all $t\in \RR$. Then, it follows that $g_{e_1}(t,y) = - \frac{a^2 \eps^4}{2\kappa} \cos(4\pi y_2) e_1$ and $g_{e_2}(t,y) = - a \eps^2 \sin(2\pi y_2) e_1$, so that the skew-symmetric matrices $\mathbf{m}_e(t,\cdot)$ are given explicitly by $\mathbf{m}_{e_1}(t,y) = \tfrac{a^2 \eps^4}{8\pi \kappa} \sin(4\pi y_2)  \Itwo^\perp$ and $\mathbf{m}_{e_2}(t,y) = - \frac{a \eps^2}{2\pi} \cos(2\pi y_2) \Itwo^\perp$.

Using the explicit computations in the previous paragraph, the equation~\eqref{e.advec.s.hom} for $\theta$ may also be written as 
\begin{equation}
\partial_t \theta - \kappa \bigl( 1 + \tfrac{a^2 \eps^4}{2\kappa^2} \bigr) \partial_{11} \theta - \kappa \partial_{22} \theta = 0 \,.
\label{eq:theta:eqn:for:shear}
\end{equation}
Moreover, 
the parabolic two-scale expansion is then given by~\eqref{e.twoscale.wep.para} as 
\begin{equation}
\label{eq:theta:tilde:for:shear}
 \tilde{\theta}_\eps(t,x):=
 \theta(t,x) + \eps   \partial_1 \theta(t,x) \chi_{e_1}\bigl(\tfrac{x_2}{\eps} \bigr)  
 = \theta(t,x) + \tfrac{a \eps^3}{2\pi \kappa} \cos\bigl(\tfrac{2\pi x_2}{\eps} \bigr)  \partial_1 \theta(t,x)\,.
\end{equation}
Additionally, the ``error'' term $f_\eps$ defined in~\eqref{e.advec.s.f} may be computed explicitly as
\begin{align}
f_\eps(t,x) 
&=-  \eps 
\ddiv \left( 
\left( \chi_{e_1} \mathbf{a} + \mathbf{m}_{e_1}  \right)\bigl(\tfrac{x_2}{\eps} \bigr)
\nabla \partial_1\theta(t,x)
+
\mathbf{m}_{e_2}\bigl(\tfrac{x_2}{\eps} \bigr)
\nabla \partial_2 \theta(t,x) 
 \right)
+ \eps 
 \chi_{e_1}\bigl(\tfrac{x_2}{\eps} \bigr) \partial_1 \partial_t \theta(t,x) 
 \notag\\
 &=-  \eps 
\ddiv \left( 
\tfrac{a \eps^2}{2\pi} \cos(\tfrac{2\pi x_2}{\eps}) \nabla \partial_1\theta(t,x)
-  \tfrac{a^2 \eps^4}{8\pi \kappa} \sin(\tfrac{4\pi x_2}{\eps})  \nabla^\perp \partial_1\theta(t,x) 
 \right)
\notag\\
& \qquad
+  \eps 
\ddiv \left( 
 \tfrac{a \eps^2}{2\pi} \cos(\tfrac{2\pi x_2}{\eps})    
\nabla^\perp \partial_2 \theta(t,x) 
 \right)
+ \eps  \partial_1 \left( 
\tfrac{a \eps^2}{2\pi \kappa} \cos(\tfrac{2\pi x_2}{\eps})  \partial_t \theta(t,x) \right)
\notag\\
&=: \ddiv v_\eps(t,x)
 \,,
 \label{eq:forcing:in:divergence:form}
 \end{align}
for an implicitly defined vector $v_\eps$, and where we recall that $\nabla^\perp = \Itwo^\perp \nabla = (-\partial_2, \partial_1)$.

In order to conclude, we at last perform some PDE estimates. 
First, we need some standard bounds for~$\theta$ which are just standard pointwise estimates for the heat equation after a change of variables. These say that, for every~$\ell\in \{ 0,1,2,\ldots\}$,  
\begin{equation}
\label{e.pointwise.bounds}
(\kappa^{-\frac12} R)^{\ell} \| (\bar{\mathbf{a}}^{\frac12} \nabla)^\ell \theta \|_{L^\infty_{x,t}(Q_R)} 
+
(\kappa^{-\frac12} R)^{2\ell}
\| \partial_t ^\ell \theta \|_{L^\infty_{x,t}(Q_R)} 
\leq 
C_\ell
\| \theta \|_{\underline{L}^2_{x,t}(Q_{2R})}
\,. 
\end{equation}
Next, we record a bound which is immediate from~\eqref{eq:forcing:in:divergence:form} and standard global energy estimates for~\eqref{eq:theta:eqn:for:shear} namely
\begin{align}
\label{e.v.eps.estimate}
\norm{v_\eps}_{L^2_{x,t}(Q_R)} 
&\les
\bigl( a\eps^3 {+} \tfrac{a^2\eps^5}{\kappa} \bigr) \! \norm{ \partial_{1}\nabla \theta}_{L^2_{x,t}(Q_R)} 
+
 a\eps^3 \! \norm{ \partial_{22} \theta}_{L^2_{x,t}(Q_R)} 
+
\tfrac{a \eps^3}{ \kappa}\! \norm{\partial_t \theta}_{L^2_{x,t}(Q_R)} 
\notag\\
&\les \tfrac{\kappa \eps}{R^2 }  \bigl(1 + \tfrac{a \eps^2}{\kappa} \bigr)
\norm{\theta}_{L^2_{x,t}(Q_{2R})}
\,,
\end{align}
where the implicit constant in the $\les$ symbol is universal. 
With this bound, we return to \eqref{eq:its:a:linear:world} and recall the parabolic problem 
\begin{equation}
\partial_t(\tilde {\theta}_\eps - \theta_\eps) - \ddiv \bigl( \mathbf{a}^\eps \nabla (\tilde{\theta}_\eps - \theta_\eps) \bigr) = \ddiv v_\eps
\,.
\label{eq:its:a:linear:world:again}
\end{equation} 
Note that  \eqref{eq:theta:tilde:for:shear} gives $(\tilde{\theta}_\eps - \theta_\eps)(t,x) = \frac{a\eps^3}{2\pi\kappa} \cos(\frac{2\pi x_2}{\eps}) \partial_1 \theta(t,x)$ for  $(t,x) \in \partial_{\mathrm{par}}Q_{R}$, 
and so $\tilde{\theta}_\eps - \theta_\eps$ does not vanish identically on $\partial_{\mathrm{par}}Q_{R}$; instead we have
\begin{align*}
\Vert \tilde{\theta}_\eps - \theta_\eps\Vert _{L^\infty_{x,t}(\partial_{\mathrm{par}}Q_{R})} 
& \les \tfrac{a \eps^3}{\kappa} \norm{\partial_1 \theta}_{L^\infty(Q_{R})}
\leq 
\tfrac{\eps}{R}  \| \theta  \|_{\underline{L}_{x,t}^2(Q_{2R})}
 \,.
\end{align*}
The right side is bounded by that of~\eqref{eq:wamboozle}. Therefore, by the maximum principle, we can correct this discrepancy at the boundary (and just ignore it for our purposes).

We continue therefore by pretending that~$\tilde{\theta}_\eps - \theta_\eps$ vanishes on~$\partial_{\mathrm{par}}Q_R$, and return to~\eqref{eq:its:a:linear:world:again} to perform a standard energy estimate. We obtain:
\begin{align*}
\| \tilde{\theta}_\eps - \theta_\eps\|_{L^\infty_t L^2_x(Q_{R})}^2
+ 
\kappa \| \nabla( \tilde{\theta}_\eps - \theta_\eps)\|_{L^2_{x,t}(Q_{R})}^2
\les \kappa^{-1} \| v_\eps \|_{L^2_{x,t}(Q_{R})}^2
\,.
\end{align*}
We continue by ignoring the gradient term on the left side and appeal to~\eqref{eq:theta:tilde:for:shear},~\eqref{e.pointwise.bounds} and~\eqref{e.v.eps.estimate} to obtain
\begin{align*}
\| \tilde{\theta}_\eps - \theta \|_{L^\infty_t \underline{L}^2_x(Q_{R})}
&
\leq
\| \tilde{\theta}_\eps - \theta_\eps\|_{L^\infty_t \underline{L}^2_x(Q_{R})} + 
\| \theta_\eps - \theta \|_{L^\infty_t \underline{L}^2_x(Q_{R})}
\notag \\ & 
\les 
\kappa^{-\frac12} (\kappa^{-1} R^2 )^{\frac12}  \| v_\eps \|_{\underline{L}^2_{x,t}(Q_{R})}
+
\tfrac{\eps}{R} \norm{\theta}_{\underline{L}^2(Q_{2R})}
\les
 (1 + \tfrac{a\eps^2}{\kappa} ) \tfrac{\eps}{R}  \norm{\theta}_{\underline{L}^2_{x,t}(Q_{2R})} 
 \,.
\end{align*}
This concludes the proof of~\eqref{eq:wamboozle}.
\end{proof}

\section{An outline of the proof of Theorem~\ref{thm:main}}
\label{sec:proof:outline}

The proof consists of two parts: the construction of multi-scale the vector field $u$ (see~\S\ref{sec:construct}), and the analysis of the passive scalars $\theta^\kappa$ via quantitive homogenization (see~\S\ref{sec:kappa:m:theta:m}--\S\ref{sec:proof:conclude}). Throughout this section, it is instructive to keep in mind the intuition gained from the discussion in~\S\ref{sec:hom}.

\subsection{Construction of the vector field \texorpdfstring{$u$}{u}}
\label{sec:construct}
The proof of Theorem~\ref{thm:main}  is the most difficult in the lowest nontrivial dimension, namely $d=2$, for the same geometric  reasons that the Onsager conjecture was hardest to establish in the two-dimensional case~\cite{GiriRadu}. We present the $d=2$ construction next.

The advantage of the case $d=2$ is notational: zero-mean divergence-free vector fields $u$ are given as the rotated gradient of a scalar stream function $\phi$ (with associated anti-symmetric matrix $\Itwo^\perp \phi$); as such, we avoid the notational care one must use when handling vectors, or when working with the $\curl$ operator.

For $m\geq 0$, we recursively construct a sequence of periodic, in $(t,x) \in[0,1] \times \TT^2  $, $C^\infty_{x,t}$ smooth vector fields $u_m = \nabla^\perp \phi_m$, and then define $u = \lim_{m\to \infty} u_m$; here the limit is taken with respect to the  $C^\alpha_{x,t}$ topology. We initialize the construction with $\phi_0=0$, so that $u_0=0$. For any $m\geq 1$, having constructed the smooth,  periodic, incompressible vector field $u_{m-1}$ with associated stream function $\phi_{m-1}$,  the induction step $m-1\mapsto m$ consists in the construction of a suitable stream function ``increment'' $\phi_{m}-\phi_{m-1}$. For this purpose, we make the following initial parameter choices:
\begin{itemize}[leftmargin=*]
\item  Let $\{ \eps_m \}_{m\geq 0}$ be a decreasing sequence of scales, with $\eps_m \to 0$ as $m\to \infty$. We set $\eps_0 = 1$, and $\eps_1 = \Lambda^{-1}$ for a large integer $\Lambda \gg 1$, chosen at the end of the proof. At each step $m\geq 1$, we view $\eps_{m}$ as the ``fast scale'' of oscillation,  and $\eps_{m-1}$ as the ``slow scale''.\footnote{The discussion in~\S\ref{sec:dispersion} refers to the fast scale $\eps$ and the slow scale $1$. As such, in this new context it is the ratio $\frac{\eps_{m}}{\eps_{m-1}} < 1 $ which plays the role of the small parameter $\eps$ in \S\ref{sec:dispersion}.}  Based on previous experience with multi-scale constructions in turbulence (see e.g.~\cite{DLSz13}), we expect that it will be technically (very!) convenient to consider the sequence $\eps_m$ to decay ``slightly super-exponentially'', meaning that 
\begin{equation}
\label{eq:eps:m:def}
\eps_{m} \simeq \eps_{m-1}^q \,, 
\qquad \mbox{for some} \qquad 
0 < q- 1 \ll 1 \,,
\end{equation}
for all $m\geq 2$.
Here, the ``$\simeq$'' is not an ``$=$'' sign because we  need $\eps_m^{-1}$ to be an integer for all $m\geq 0$ (in order to enforce $\TT^2$--periodicity). How small the positive parameter $ q-1$ needs to be taken is chosen as part of the proof.

\item We let $a_{m}\gg 1$ denote the Lipschitz norm (up to universal constants) of the stream function increment $\phi_{m}-\phi_{m-1}$. If we think of the increment $\phi_{m}-\phi_{m-1}$ as oscillating at frequency $\eps_{m}$, then in order to ensure uniform boundedness of the velocity fields $\{u_m\}_{m\geq 0}$ in $C^0_t C^\alpha_x$, it is natural and necessary to define 
\begin{equation}
\label{eq:a:m:def}
a_{m} := \eps_{m}^{\alpha -1} \,, \qquad \mbox{for all} \qquad m \geq 0 \,.
\end{equation}
Note that since $\alpha < 1$, we have $a_m \to \infty$ as $m\to \infty$.
\end{itemize}

\subsubsection{A naive attempt for all \texorpdfstring{$0< \alpha < 1$}{0<alpha<1}}
\label{sec:naive:construction}
With $\{\eps_m\}_{m\geq 0}$ and $\{a_m\}_{m\geq 0}$ fixed according to \eqref{eq:eps:m:def}--\eqref{eq:a:m:def}, and keeping in mind the result of~Lemma~\ref{lem:Taylor:quantitative}, a ``naive'' attempt at defining the stream function increment $\phi_{m}-\phi_{m-1}$ is as follows.

Based on~\eqref{eq:psi:a:eps:def}  we are tempted to declare that the stream function increment is given by $a_{m} \eps_{m}^2 \sin(\frac{2\pi x_2}{\eps_{m}})$. Note however that in this case the homogenized matrix~\eqref{eq:bf:a:bar:shear} presents an enhancement only in the $e_1 \otimes e_1$ component, which is undesirable. 
The natural fix to this issue is to alternate the direction of shearing, by considering a stream function which sequentially in time alternates the stream functions $a_{m} \eps_{m}^2 \sin(\frac{2\pi x_2}{\eps_{m}})$ and $a_{m} \eps_{m}^2 \sin(\frac{2\pi x_1}{\eps_{m}})$, ``glued'' by a partition of unity in time, with characteristic time-scale proportional to some $0< \tau_{m}  \ll 1$, to be determined. We would like however these shears acting on different directions to not interact/interfere with each other (see also~\cite{Daneri,Isett,BDLSZV}), and hence we leave room in which no shearing occurs, whenever we switch directions. As such, we define 
\begin{equation*}
\Psi_k(x) := \begin{cases}
\sin(2\pi x_1)\,, & k \equiv 1 \mbox{ mod } 4 \,,\\
\sin(2\pi x_2)\,, & k \equiv 3 \mbox{ mod } 4\,,\\
0 \,, & k \equiv 0 \mbox{ mod } 2 \,,
\end{cases} 
\end{equation*}
and we attempt to declare
\begin{equation}
\label{eq:phi:increment:bad}
\phi_{m}(t,x) - \phi_{m-1}(t,x) 
:= \sum_{k\in \ZZ} 
\zeta \Bigl(\tfrac{t}{\tau_{m}}-k \Bigr) 
a_{m} \eps_{m}^2 
\Psi_k \Bigl(\tfrac{x}{\eps_{m}} \Bigr)
\,,
\end{equation}
where $\mathbf{1}_{[-\frac 13, \frac 13]} \leq \zeta \leq \mathbf{1}_{[-\frac 2 3, \frac 23]}$ is a $C^\infty$ smooth unit-scale bump function with $\int_{\RR} \zeta^2(t) dt  = 1$, and which induces a partition of unity via  $\sum_{k\in \ZZ} \zeta(\cdot - k) \equiv 1$.

By construction, the stream function increment in~\eqref{eq:phi:increment:bad} is $4 \tau_{m}$--periodic in time, and since the formula~\eqref{eq:abar:def} for the homogenized matrix  involves averaging in time, we expect that enhancement will occur in both the $e_1 \otimes e_1$ and $e_2 \otimes e_2$ components, as desired, with the  ``amount of enhancement'' being 25\% of the one in~\eqref{eq:bf:a:bar:shear}. Note however that in order to see this enhancement, Lemma~\ref{lem:Taylor:quantitative} dictates that we ``zoom out''  in both time and space; in turn, this imposes restrictions on $\tau_{m}$ and $\eps_{m}$.

The first condition in~\eqref{eq:scaling:stitches} dictates a constraint on the time scale $\tau_{m}$, which determines the periodicity in time (inverse frequency) of the stream function increment in~\eqref{eq:phi:increment:bad}. We need the amount of time that each of the stream functions $\Psi_k$ are active to be long enough to allow for homogenization to be observed; this dictates:
\begin{equation}
\label{eq:tau:m:constraint}
\tau_{m} 
\gg  
a_{m}^{-1} \cdot \bigl( 1 + \tfrac{a_{m} \eps_{m}^2}{\kappa_{m}} \bigr)  
=
a_{m}^{-1} + \tfrac{\eps_{m}^2}{\kappa_{m}}
\geq 
\tfrac{\eps_{m}^2}{\kappa_{m}}
\,.
\end{equation}
The second condition in~\eqref{eq:scaling:stitches} dictates that the length scale corresponding to ``slow oscillations'' (meaning, the scale of oscillation of the ``slow'' stream function $\phi_{m-1}$),  is large enough for homogenization to be observed; this dictates:
\begin{equation}
\label{eq:eps:m:constraint}
\eps_{m-1} 
\gg  
\eps_{m} \cdot \bigl( 1 + \tfrac{a_{m} \eps_{m}^2}{\kappa_{m}} \bigr) 
\geq \tfrac{a_{m} \eps_{m}^3}{\kappa_{m}}
\,.
\end{equation}

With~\eqref{eq:tau:m:constraint}--\eqref{eq:eps:m:constraint}, Lemma~\ref{lem:Taylor:quantitative} indicates that upon ``zooming out'' from scale $\eps_{m}$ to scale $\eps_{m-1}$,\footnote{It is this $m \mapsto m-1$ homogenization step which goes from small scales to larger ones, and this is the motivation for referring to the cascade of effective diffusivities as being   ``backwards''.}  the enhancement of diffusion is given by $\kappa \Itwo \mapsto \kappa \Itwo + \frac 18 \kappa  \cdot (\frac{a_{m} \eps_{m}^2}{\kappa})^2  \Itwo $; the division  by $4$  when compared to~\eqref{eq:bf:a:bar:shear} is due to time-averaging.

If the above-described  scenario can be justified rigorously, iteratively in $m$, then we have set up the parameters $\{\eps_m\}_{m\geq 0}$, $\{a_m\}_{m\geq 0}$, and $\{\tau_m\}_{m\geq 0}$, such that when ``coarse-graining'' at scale $\eps_{m}$\footnote{See~\cite{AK24,ABK24} for a general theory of \emph{quantitative theory of coarse-graining} for elliptic and parabolic PDEs with multi-scale oscillating coefficients.} the advection-diffusion equation~\eqref{eq:drift:diffusion}, with vector field $u = \nabla^\perp \phi$, with $\phi$ defined by summing over $m\geq 0$ the increments in~\eqref{eq:phi:increment:bad}, experiences an effective diffusivity $\kappa_m$ (see~\S\ref{sec:kappa:m:theta:m}), which satisfies the recursion relation 
\begin{equation}
\label{eq:kappa:m:recursion}
\kappa_{m-1}:= \kappa_{m} + \tfrac 18 \kappa_{m}  \cdot \bigl(\tfrac{a_{m} \eps_{m}^2}{\kappa_{m}}\bigr)^2
\,.
\end{equation}
The recursion~\eqref{eq:kappa:m:recursion} is to be ``initialized'' with $\kappa_M = \kappa$, for an integer $M \gg 1$ which is to be determined such that $\eps_M$ is proportional to the dissipation length-scale associated with the advection-diffusion equation~\eqref{eq:drift:diffusion} with a $C^\alpha_{x,t}$ vector field $u$. With this initialization, the recursion~\eqref{eq:kappa:m:recursion} is to be solved backwards: $M \to M-1 \to \ldots \to m \to m-1 \to \ldots \to 0$, with the hope that $\kappa = \kappa_M < \kappa_{M-1} < \ldots < \kappa_{m} < \kappa_{m-1} < \ldots < \kappa_0$, and that $\kappa_0=O(1)$ as $\kappa\to 0$. 

Let us next inquire whether the parameter choices~\eqref{eq:eps:m:def}, \eqref{eq:a:m:def}, together with the recursion~\eqref{eq:kappa:m:recursion}, are consistent with the constraints~\eqref{eq:tau:m:constraint}, \eqref{eq:eps:m:constraint}, and with the fact that $\kappa_0 = O(1)$.  Assuming for simplicity that $\kappa = \kappa_M = \eps_M^{1+\alpha+\gamma}$ for some $M\geq 1$\footnote{That is, we are letting $\kappa = \kappa_M \to 0$ along the designated sequence $\eps_M^{1+\alpha+\gamma} \to 0$, as $M\to \infty$.}; then, it is not hard to verify that the relation
\begin{equation}
\label{eq:kappa:m:explicit}
\kappa_m \simeq \eps_m^{1+\alpha+\gamma}\,,
\qquad 
\mbox{where}
\qquad
\gamma = (q-1) \tfrac{1+\alpha}{q+1} \,,
\end{equation}
for all $M \geq m \geq 0$, 
is consistent with \eqref{eq:kappa:m:recursion}, where we recall that $0< q-1 \ll 1$. Moreover, setting $m=0$ in \eqref{eq:kappa:m:explicit} yields $\kappa_0 \simeq 1$, as desired, independently of the value of $M \gg 1$ (and hence of $\kappa\ll 1$). 

The constraint~\eqref{eq:tau:m:constraint} is then used as to pick a time-scale that satisfies
\begin{equation}
\tau_m \gg a_m^{-1} + \eps_m^2 \kappa_m^{-1} = \eps_m^{1-\alpha} + \eps_m^{1-\alpha-\gamma} \simeq \eps_m^{1-\alpha-\gamma} \,.
\label{eq:tau:m:lower:bound}
\end{equation}
Such a choice is permissible if $\tau_m\leq 1$\footnote{The choice of $a_m$ in \eqref{eq:a:m:def}, together with the definition~\eqref{eq:phi:increment:bad} implies the uniform-in-$m$ space regularity estimate of the velocity increments: $\| u_m - u_{m-1}\|_{C^0_t C^{\alpha}_x} = \|\nabla^\perp (\phi_m - \phi_{m-1})\|_{C^0_t C^\alpha_x} \simeq a_m \eps_m^{1-\alpha} \simeq 1$. On the other hand, if we wish to estimate uniform-in-$m$ time regularity of the velocity increments, we see that $\| u_m - u_{m-1}\|_{C^{\alpha}_t C^{0}_x} = \|\nabla^\perp (\phi_m - \phi_{m-1})\|_{C^\alpha_t C^0_x} \simeq \tau_m^{-\alpha} a_m \eps_m  \simeq \tau_m^{-\alpha} \eps_m^\alpha \ll \eps_m^{\alpha(\alpha+\gamma)}$, where in the last inequality we have used~\eqref{eq:tau:m:lower:bound}. For any $\alpha>0$ and $q>1$ we thus obtain $C^\alpha_t C^0_x$--regularity of our vector field for free.} for all $m\geq 0$, i.e., as soon as $1-\alpha-\gamma >0$. In light of the definition of $\gamma$ in \eqref{eq:kappa:m:explicit}, this is indeed possible \emph{for all values $\alpha\in(0,1)$}, by letting $1 < q < \frac{1}{\alpha}$. Note that as $\alpha \to 1^-$, this implies $q \to 1^+$.
 
\subsubsection{What goes wrong with this naive construction?} 
\label{sec:mean:flow:killer}
The discussion in \S\ref{sec:naive:construction} indicates that for any $\alpha \in (0,1)$ we may construct a $C^{\alpha}_{x,t}$ vector field (with stream function obtained by summing \eqref{eq:phi:increment:bad} in $m\geq 1$), such that the sequence of renormalized diffusivities are compatible with the recursion relation~\eqref{eq:kappa:m:recursion}, and the space-time separation is compatible with~\eqref{eq:tau:m:constraint} and~\eqref{eq:eps:m:constraint}, as dictated by Lemma~\ref{lem:Taylor:quantitative}; additionally, the amount of diffusion experienced by the passive scalar when ``coarse-grained'' at $O(1)$ scales is itself $O(1)$, independently of the value of the molecular diffusion $\kappa \ll 1$. 

In discussing the above heuristic, \emph{a subtle (and fatal) point was overlooked}: the homogenization of a shear flow discussed in Lemma~\ref{lem:Taylor:quantitative} \emph{assumed that the shear flow $u = u_{a,\eps}$ had zero-mean} when averaged over one periodic cell of the fast variable $y=\frac{x}{\eps}$. Indeed, the $\cos(2\pi y_2)$ present in the definition of $u_{a,\eps}$ (see~\eqref{eq:u:a:eps:def}) ensures that this velocity field has zero-mean over $\TT^2(dy)$. It was the very fact that $\langle u_{a,\eps} \rangle = 0$ that allowed us to write this velocity field as the rotated gradient of a periodic stream function $\psi_{a,\eps}$ (see~\eqref{eq:psi:a:eps:def}), which was at the core of our homogenization analysis.

Returning to the construction in \S\ref{sec:naive:construction}, for each  $m \geq 1$ we note that the velocity field $u_m$ \emph{does not have zero-mean when averaged over a periodic cell of the fast variable} $\frac{x}{\eps_m}$. Indeed, the ``macroscopic'' velocity field $u_{m-1}$ \emph{is essentially a constant} from the point of view of the ``microscopic'' velocity increment $u_m - u_{m-1}$, which is $(\eps_m \TT)^2$--periodic. To be precise,    recursion~\eqref{eq:phi:increment:bad} shows that we have $\fint_{(\eps_m \TT)^2} (u_m-u_{m-1})(t,\cdot) = 0$ for all $t\in \RR$, while for any $1 \leq j \leq m-1$, we have\footnote{Here we use $\fint_{0}^{\eps_m} \cos(\frac{2\pi x_1}{\eps_{m-j}}) dx_1 = \eps_{m-j}^{\alpha+1} (2\pi \eps_m)^{-1}  \sin(\frac{2\pi \eps_m}{\eps_{m-j}}) \simeq \eps_{m-j}^\alpha$, since $\frac{\eps_m}{\eps_{m-j}} \ll 1$.} $\fint_{(\eps_m \TT)^2} (u_{m-j}-u_{m-j-1})(t,\cdot) \simeq  \eps_{m-j}^\alpha \neq 0$, for    $t \in \cup_{k\in 2 \ZZ + 1} [\tau_{m-j} (k - \frac 13),\tau_{m-j} (k+\frac 13)]$. Note that $\eps_{m-j}^\alpha \gg \eps_m^\alpha$ for any $1\leq j \leq m-1$, so not only is the macroscopic velocity field $u_{m-1}$ essentially a constant from the point of view of microscopic oscillations, in relative terms it is a \emph{large constant}.

The issue is that presence of a large constant drift velocity, superimposed on a zero-mean shear flow, is able to remove most of the enhancement of the diffusivity; this is so because generically this constant drift will point in an irrational/Diophantine direction, causing \emph{averaging to overtake homogenization}. This is explained quantitatively in~\cite[Appendix~A]{AV24}. In the physics literature this phenomenon---that a non-constant mean flow interferes and potentially destroys Taylor diffusion---is termed the ``sweeping effect''~\cite{AvM2a,MK,EB}. To sum up, the construction presented in  \S\ref{sec:naive:construction} does not work, at least not ``as is''.

\subsubsection{Small-scale features are transported by the large-scale flow: \texorpdfstring{$\alpha < \frac 13$}{alpha<1/3}}
\label{sec:construct:Lagrangian}
To overcome the issue raised at the end of~\S\ref{sec:mean:flow:killer} we recall one of the cornerstones of phenomenological turbulence theories, the observation that microscopic ``eddies'' are transported by the surrounding macroscopic flow, for a suitable amount of time~\cite{Frisch}. 

Roughly speaking, the idea put forth in~\cite{AV24}\footnote{In the context of the Onsager conjecture, this idea (a proto-version of it) was introduced first in~\cite{DLSz13}, and played a major role in all subsequent works, leading to~\cite{Isett,BDLSZV,NovackVicol,GiriRadu,GiriKwonNovack}. In the context of homogenization of advection-diffusion equations, a version of this idea appeared in~\cite{MPP}.}  is that the microscopic/fast shear flows should \emph{not} be naively superimposed on top of the existing macroscopic/slow shear flows (as~\eqref{eq:phi:increment:bad} dictates); \emph{instead} the microscopic/fast shear flows should be added in a frame of reference that moves according to the macroscopic/slow velocity field (which is to say, in the Lagrangian coordinates of this slow velocity), so that in this frame the shear flows continue to have zero mean (to leading order).

For this purpose, for each $[0,1]\times\TT^2$--periodic, divergence-free  velocity field $u_m$, we define the associated flow maps $X_m = X_m(t,x;s)$ by solving the ODE
\begin{equation*}
\tfrac{d}{dt} X_m(t,x;s) = u_m(t,X_m(t,x;s)) \,,
\qquad 
X_m(s,x;s) = x \,,
\end{equation*}
where $t,s \in [0,1]$ and $x\in \TT^2$.  Here we use the notation introduced in~\S\ref{sec:construct}; in particular, for each $m\geq 1$ the velocity fields $u_m \in C^\infty_{x,t}$. We shall denote by $X_m^{-1} = X_m^{-1}(x,t;s)$ the corresponding inverse maps (sometimes referred to as the ``back-to-labels map''), satisfying the relations $X_m^{-1}(t,X_m(t,x;s) ;s) = x$ and $X_m(t,X_m^{-1}(t,x;s);s) = x$.\footnote{Note that the maps $x \mapsto X_m(t,x;s)$ and $x \mapsto X_m^{-1}(t,x;s)$ are volume-preserving (since $\ddiv u_m = 0$), while the maps $x \mapsto X_m(t,x;s)-x$ and $x \mapsto X_m^{-1}(t,x;s)-x$ are $[0,1]\times\TT^2$--periodic (since so is $u_m$).} Classical ODE theory dictates that the Lagrangian maps $x\mapsto X_m(\cdot,x;\cdot)$ and their inverse flows $x\mapsto X_m^{-1}(\cdot,x;\cdot)$ remain close to the identity map, in the sense that $|\nabla X_m(x,t;s) - \Itwo| + |\nabla X_m^{-1}(x,t;s) - \Itwo| \ll 1$, only for time intervals which are small multiples of the Courant-Friedrichs-Lax time, namely for $|t-s| \ll \| \nabla u_m \|_{L^\infty_{x,t}}^{-1}$. Thus, the usage of Lagrangian coordinates necessitates the introduction of a new time-scale $\{ \tau_{m}^{\prime\prime} \}_{m\geq 1}$, which needs to satisfy the relation\footnote{We note here a difference with the notation in~\cite{AV24}: there, $\tau_m^{\prime\prime}$ is chosen to be small with respect to the CFL time-scale of $u_{m-1}$, i.e.~$\tau_m^{\prime\prime} a_{m-1} \ll 1$. In this note we choose the slightly more intuitive notation, which relates $\tau_m^{\prime\prime}$ to the CFL time-scale of $u_m$, i.e.~$\tau_m^{\prime\prime} a_{m} \ll 1$; see~\eqref{eq:tau:m:cond:1}.}
\begin{equation}
\label{eq:tau:m:cond:1}
\tau_{m}^{\prime\prime} \|\nabla u_m \|_{L^\infty_{x,t}} \ll 1
\qquad 
\mbox{or, equivalently,}
\qquad
\tau_{m}^{\prime\prime} a_m = \tau_{m}^{\prime\prime} \eps_m^{\alpha-1} \ll 1
\,.
\end{equation}
According to this new time-scale, for every $\ell \in \ZZ$ we denote\footnote{As before, we have a difference with the notation with respect to~\cite{AV24}: here we write $X_{m}(\cdot,\cdot; \ell \tau_m^{\prime\prime})$, whereas in~\cite[Equation~(2.36)]{AV24} we wrote $X_{m-1}(\cdot,\cdot; \ell \tau_m^{\prime\prime})$.}
\begin{equation*}
X_{m,\ell}(t,x):= X_m(t,x;\ell\tau_m^{\prime\prime})\,,
\qquad \mbox{and} \qquad 
X_{m,\ell}^{-1}(t,x):= X_m^{-1}(t,x;\ell\tau_m^{\prime\prime})\,,
\end{equation*}
the forward and backward flows of $u_m$, initialized at time $\ell \tau_m^{\prime\prime}$.

With the above notation, for $m\geq 1$ we define the stream function increment by\footnote{If one wishes to prove Theorem~\ref{thm:main}, the parameters $a_m$ appearing in~\eqref{eq:phi:increment:good} need to become ``slow functions'' of space and time, oscillating on scales $\eps_{m-1}$ in $x$ and $\tau_{m-1}$ in $t$. The purpose of these slow functions in a convex-integration scheme is to cancel (through the mean of their squares) a leftover Euler-Reynolds stress error. We refer the reader to~\cite[Section 5]{BDLSZV} and ~\cite[Section 7.3]{Leipzig} for these details.}
\begin{equation}
\label{eq:phi:increment:good}
\phi_{m}(t,x) 
-\phi_{m-1}(t,x) 
:=
\sum_{k,\ell \in \ZZ} 
\zeta\left(\tfrac{t}{\tau_{m}}-k \right) 
\hat{\zeta}\left(\tfrac{t}{\tau_{m-1}^{\prime\prime}}-\ell \right) 
a_{m} \eps_{m}^2 
\Psi_k \Bigl(\tfrac{X_{m-1,\ell}^{-1}(t,x)}{\eps_{m}} \Bigr)
\,,
\end{equation}
where $\hat{\zeta}$ is another $C^\infty$ smooth unit-scale bump function, which smoothly approximates the cutoff function $\mathbf{1}_{[-\frac 12, \frac 12]}$, and such that  its integer shifts nearly form a partition of unity: $\sum_{\ell \in \ZZ}\hat{\zeta}(\cdot-\ell) \simeq 1$.\footnote{For technical reasons, the time-cutoff functions $\zeta$ and $\hat{\zeta}$ are chosen to satisfy a number of fine properties not mentioned here, in order to deal with the fast switching of shear flows and the slow switching of flow maps; see~\cite[Section 2.1]{AV24} for precise descriptions.} Note that the cutoff functions $\zeta(\cdot - k)$, the parameters $\eps_m$, $a_m$, $\tau_m$, and also the alternating shears $\Psi_k$ appearing in~\eqref{eq:phi:increment:good} are the same as in~\eqref{eq:phi:increment:bad}. At this stage we only assume that the time-scale $\tau_{m-1}^{\prime\prime}$ satisfies the smallness condition from~\eqref{eq:tau:m:cond:1}, but as we shall see next, a largeness constraint emerges (see~\eqref{eq:tau:m:cond:2} below). With the new definition of the stream function increment in~\eqref{eq:phi:increment:good} replacing the previous naive attempt from~\eqref{eq:phi:increment:bad}, the velocity increment is defined again as as $u_m - u_{m-1} = \nabla^\perp (\phi_m - \phi_{m-1})$. 

The usefulness of~\eqref{eq:phi:increment:good} is that on the support of each cutoff function $\hat{\zeta}(\frac{\cdot}{\tau_{m-1}^{\prime\prime}} -\ell)$, the ``Lagrangian increment'' $(\phi_m-\phi_{m-1})(\cdot,X_{m-1,\ell}(\cdot,x))$ behaves as $a_m \eps_m^2 \Psi_k(\frac{x}{\eps_m})$, which has zero-mean with respect to the fast space variable $\frac{x}{\eps_m}$. This overcomes the issue raised at the end of~\S\ref{sec:mean:flow:killer}. However, in order to use the same heuristics as in~\S\ref{sec:naive:construction}, leading to the recursion relation~\eqref{eq:kappa:m:recursion} and asymptotic~\eqref{eq:kappa:m:explicit} for the cascade of renormalized diffusivities, we need to ensure that the Lagrangian flow maps $X_{m-1,\ell}^{-1}(\cdot,\cdot)$ and the new time cutoffs $\hat{\zeta}(\frac{\cdot}{\tau_{m-1}^{\prime\prime}} -\ell)$ appearing in~\eqref{eq:phi:increment:good}, do not interfere with  the quantitative homogenization picture painted in~\S\ref{sec:naive:construction}.  This necessitates that the new periodic time oscillation which we have introduced (on time-scale $\tau_{m-1}^{\prime\prime}$) \emph{does not interfere}\footnote{More than that, this new time oscillation must itself be homogenized!} with the homogenization problem which produced~\eqref{eq:kappa:m:recursion}; put differently, this necessitates that from the point of view of a function oscillating at time-scale $\tau_m$, functions varying on time-scale $\tau_{m-1}^{\prime\prime}$ are essentially a constant, resulting in the new parameter constraint 
\begin{equation}
\label{eq:tau:m:cond:2}
\tau_{m-1}^{\prime\prime} \gg \tau_m . 
\end{equation}
Together, the constraint on the new time parameter $\tau_{m}^{\prime\prime}$ are
\begin{equation}
\label{eq:all:time:constraints}
\tau_m
\ll 
\tau_{m-1}^{\prime\prime}
\ll
a_{m-1}^{-1} 
\,.
\end{equation}
We also recall that $\tau_m$ needs to be chosen to satisfy~\eqref{eq:tau:m:constraint}.

At last, we try to see if we can choose the parameters in order  satisfy all of the above-mentioned constraints. The two inequalities in~\eqref{eq:all:time:constraints} indicate that a permissible choice for $\tau_{m-1}^{\prime\prime}$ can be made if and only if $\tau_m$ is chosen to satisfy $a_{m-1}^{-1} \gg \tau_m$. In turn, the lower bound for $\tau_m$ established previously in~\eqref{eq:tau:m:lower:bound} shows that such a permissible choice for $\tau_m$ is possible if and only if $a_{m-1}^{-1} \gg \eps_m^{1-\alpha-\gamma}$, where we recall that $\gamma$ is defined by~\eqref{eq:kappa:m:explicit}. Recalling the definition of $a_{m-1}$, and the $\eps_m$ to $\eps_{m-1}$ relation from~\eqref{eq:eps:m:def}, we arrive at the parameter constraint
\begin{equation}
\label{eq:alpha:one:third:emerges} 
\eps_m^{\frac{1-\alpha}{q}} 
\simeq 
\eps_{m-1}^{1-\alpha} 
=
a_{m-1}^{-1}
\gg 
\eps_m^{1-\alpha-\gamma} 
\,.
\end{equation}
Inequality \eqref{eq:alpha:one:third:emerges}  may be satisfied for all $m\geq 1$ if $\eps_1 = \Lambda^{-1}$ is chosen to be sufficiently small (to absorb the implicit constant in the $\gg$ symbol), and if $\frac{1-\alpha}{q} < 1-\alpha-\gamma$. Recalling that $\gamma = (q-1) \tfrac{1+\alpha}{q+1}$, that $\alpha \in (0,1)$ and $q>1$, \eqref{eq:alpha:one:third:emerges}  thus necessitates: 
\begin{equation*}
\tfrac{1-\alpha}{q} < 1-\alpha-\gamma
\Leftrightarrow
q \gamma  < (q-1)(1-\alpha)
\Leftrightarrow
\tfrac{q(1+\alpha)}{q+1}
< 1-\alpha
\Leftrightarrow
\alpha < \tfrac{1}{2q+1} \,.
\end{equation*}
Since $q>1$, we arrive at the \emph{Onsager-supercriticality constraint} $\alpha < \frac 13$ present in the statement of Theorem~\ref{thm:main}.

\subsubsection{The construction of \texorpdfstring{$u$}{u}: order of choosing parameters}
To sum up, for a given $\alpha \in (0,\frac 13)$, we first choose $q>1$ so that $(2q+1) \alpha <1$. Then, we define $\{\eps_m\}_{m\geq 0}$ according to~\eqref{eq:eps:m:def}, and $\{a_m\}_{m\geq 0}$ according to~\eqref{eq:a:m:def}. In light of the lower bound $\tau_m\gg \eps_m^{1-\alpha-\gamma}$ (expected from~\eqref{eq:tau:m:lower:bound}), and the upper bound $\tau_m\ll a_{m-1}^{-1}$ (implied by~\eqref{eq:all:time:constraints}), we choose $\tau_m^{-1} \gg 1$ to be an integer lying in the interval $(\eps_{m-1}^{\alpha-1}, \eps_m^{\alpha-1+\gamma})$.\footnote{\label{foot:tau:m} The choice $\tau_m^{-1} \in \mathbb{N}$ is made in order to ensure $[0,1]$--time periodicity of the respective time cutoffs. For instance, $\tau_{m}^{-1} \simeq \eps_{m}^{\frac{(\alpha-1)(q+1) + \gamma q}{2q} } $ is an allowable choice.} With $\tau_m$ chosen, we then return to~\eqref{eq:all:time:constraints} and let $(\tau_{m-1}^{\prime\prime})^{-1}$ to be an integer lying in the interval $(a_{m-1}, \tau_m^{-1})$.\footnote{For instance, an integer which satisfies $(\tau_{m}^{\prime\prime})^{-1} \simeq  \eps_{m}^{\frac{(\alpha-1)(q+3) + \gamma q}{4} }$ is an allowable choice.} With these choices, the stream functions  $\{\phi_m\}_{m\geq 0}$ are defined by the telescoping relation~\eqref{eq:phi:increment:good}. The parameter $\Lambda = \eps_1^{-1} \gg 1$ is chosen last, to absorb all $m$-independent constants appearing in our estimates.

The vector field $u$ is obtained as $u = \lim_{m\to\infty} u_m$, the limit being taken with respect to the $C^{\alpha^\prime}_{x,t}$ topology, for any $\alpha^\prime< \alpha$. Here, $u_m = \nabla^\perp \phi_m$, $\phi_m = \sum_{1\leq m^\prime\leq m} (\phi_m - \phi_{m-1})$, $\phi_0 = 0$, and the stream function increment is given by~\eqref{eq:phi:increment:good}. 

\begin{remark}[Genericity of the vector fields $u$ from Theorem~\ref{thm:main}]
\label{rem:dense}
The choice $\phi_0$ is arbitrary, made in~\cite{AV24} only for notational simplicity. The proof of Theorem~\ref{thm:main} works ``as is'' if we replace $\phi_0$ by any sufficiently smooth, zero-mean periodic stream function. Then, all we have to do  is to ensure that $\eps_1 = \Lambda^{-1}$ is chosen to be (much) smaller than the smallest scale of oscillation of this chosen smooth $\phi_0$. By telescoping the available bounds on $\|\phi_m - \phi_{m-1}\|_{C^1_{x,t}}$ (see~\cite[Proposition 2.2]{AV24}), we may then show that $\|\phi - \phi_0 \|_{C^1_{x,t}} = \| u - \nabla^\perp \phi_0\|_{C^0_{x,t}}\leq C \Lambda^{-\alpha}$, where $C= C(\alpha,q)>0$; this upper bound may be made arbitrarily small by choosing $\Lambda$ sufficiently large, showing that the class of vector fields $u$ for which Theorem~\ref{thm:main} holds is in fact \emph{dense}, in the class of all smooth divergence-free zero-mean vector fields, with respect to the $C^0_{x,t}$ topology. A stronger form of genericity of the family of vector fields $u \in C^{\alpha}_{x,t}$ which induce anomalous diffusion is obtained in~\cite{Leipzig}, see~Theorem~\ref{thm:Leipzig} above.
\end{remark}

One may show, see~\cite[Section 2.3]{AV24} that the family $\{\phi_m\}_{m\geq 0}$ is uniformly bounded in $C^{1,\alpha}_{x,t}$, that for each $m$ the stream function $\phi_m$ is real-analytic in space (uniformly in time), with radius of analyticity proportional to $\eps_m$, that the flow maps $X_m$ and their inverses $X_m^{-1}$ remain close to the identity on their natural timescale $\tau_m^{\prime\prime}$, and that they are also real-analytic with radius of analyticity proportional to $\eps_m$. That is to say, the family of stream functions $\{\phi_m\}_{m\geq 0}$ behave in the way we expect them to behave. 

\begin{remark}[Propagation of smoothness in the construction]
In~\cite{AV24} we have chosen to prove the quantitative real-analyticity of the stream functions $\phi_m$, of the Lagrangian flows $X_m$, and of their inverses $X_m^{-1}$, simply because it is true: the $\Psi_k$ are sine waves, which are certainly analytic functions, the Lagrangian flow maps of an analytic function are analytic, adding and composing analytic functions to each other preserves analyticity; one has to be careful however when estimating the  radius of analyticity of the resulting objects~\cite[Appendix B]{AV24}. This real-analyticity is however not essential for the proof: all that is needed is ``sufficiently high regularity'', to justify certain high-order expansions in the proof. Instead of working with real-analytic building blocks $\Psi_k$ (sine waves), it is sufficient to consider smooth ones, and then to add a ``mollification step'' to the proof, as is common in convex-integration schemes (see~e.g.~\cite[Section 2.4]{BDLSZV}, \cite[Section 5.3]{Buck2020}). This perspective needs to be taken if one is to work with building blocks $\Psi_k$ which are compactly supported in space, and this is precisely the perspective taken in the proof of Theorem~\ref{thm:Leipzig}, see~\cite[Section 7.1]{Leipzig}.
\end{remark}

\subsection{A backwards cascade of eddy diffusivities and homogenized problems}
\label{sec:kappa:m:theta:m}

With the vector field $u$ constructed in~\S\ref{sec:construct} it remains to rigorously justify our scale-by-scale homogenization picture, which dictates that when ``coarse-grained'' at scale $\eps_{m}$, the advection-diffusion equation~\eqref{eq:drift:diffusion}  with vector field $u$ experiences an effective diffusivity $\kappa_m$, which satisfies the recursion relation~\eqref{eq:kappa:m:recursion}, initialized with $\kappa_M = \kappa$ for a suitably chosen $M \geq 0$. This requires the introduction of a \emph{permissible set of diffusivities} (the set $\mathcal{K}$ mentioned in Theorem~\ref{thm:main}), and a  definition of ``coarse-graining'' (at least, as it pertains to the analysis discussed here).

In spite of the inherent instability of the recursion formula $\kappa \mapsto \kappa + \frac{a_m^2 \eps_m^4}{8 \kappa}$, see~\eqref{eq:kappa:m:recursion}, we have seen earlier in~\eqref{eq:kappa:m:explicit} that the relation $\kappa_m \simeq \eps_m^{1+\alpha+\gamma}$ is consistent with stably solving~\eqref{eq:kappa:m:recursion}. As such, we define the permissible set of diffusivities by
\begin{equation*}
 \mathcal{K}: = \cup_{m\geq 0} \bigl[ \tfrac 12 \eps_m^{1+\alpha+\gamma}, 2 \eps_m^{1+\alpha+\gamma}\bigr] \,.
\end{equation*}
For any given $\kappa \in \mathcal{K}$, we may identify a unique integer $M\geq 0$ such that 
\begin{equation}
\label{eq:kappa:m:real:1}
\kappa \in \bigl[ \tfrac 12 \eps_M^{1+\alpha+\gamma}, 2 \eps_M^{1+\alpha+\gamma}\bigr] \,,
\quad \mbox{and we define} \quad 
\kappa_M := \kappa \,.
\end{equation}
Then, for $m \in \{1, \ldots,M\}$, we  recursively define the \emph{renormalized diffusivities}
\begin{equation}
\label{eq:kappa:m:real:2}
\kappa_{m-1} :=  \kappa_m \Bigl( 1 + c_*\bigl ( \tfrac{a_m \eps_m^2}{ \kappa_m} \bigr)^2 \Bigr) \,.
\end{equation}
Here, $c_* > 0$ is a universal constant, which may differ from the $\frac 18$ value presented in~\eqref{eq:kappa:m:recursion} depending on the $L^2_t$--normalization of the time cutoffs $\zeta$ and $\hat{\zeta}$.

Note that~\eqref{eq:kappa:m:real:1}--\eqref{eq:kappa:m:real:2} defines a \emph{backwards cascade of eddy diffusivities},\footnote{It is verified in~\cite[Lemma 3.4]{AV24} that $\kappa_m \in \mathcal{K}$ for all $m \in \{0,\ldots,M\}$.} which is $O(1)$ at the unit scale, in the sense that
\begin{equation*}
     \kappa = \kappa_M \ll \kappa_{M-1} \ll ...\ll \kappa_m \ll \kappa_{m-1}  \ll ... \ll \kappa_0 \in [\tfrac 12,2] \,.
\end{equation*}
The above relation holds no matter how small the molecular diffusivity $0<\kappa\ll 1$ is; except that as $\kappa \to 0$ we have $M\to \infty$, and thus in this limit we need to perform infinitely many renormalization/quantitative homogenization steps. 

It is useful to keep in mind the relations 
\begin{equation*}
\tfrac{a_m \eps_m^2}{\kappa_m} 
\simeq 
\eps_m^{-\gamma} \gg 1 \,,
\qquad  \qquad 
\tfrac{\eps_m^2}{\kappa_m\tau_m} 
\simeq
\eps_{m}^{ \frac{(1-\alpha (1+2 q) )(q-1) }{2q(q+1)}  }
\ll 1 \,,
\end{equation*}
which hold since $ \kappa_m \simeq \eps_m^{1+\alpha+\gamma} \simeq a_m \eps_m^2 \cdot \eps_m^\gamma$, since $\tau_m$ is given by the asymptotic in Footnote~\ref{foot:tau:m}, and since $\alpha (2q+1) <1$.

Having defined the permissible set of diffusivities $\mathcal{K}$, and the sequence $\{\kappa_m\}_{m=0}^{M}$ (see~\eqref{eq:kappa:m:real:1}--\eqref{eq:kappa:m:real:2}), it thus remains to define what we mean by  ``coarse-graining'' the advection-diffusion equation~\eqref{eq:drift:diffusion}  at scale $\eps_m$. A precise and robust definition of this concept, in a much more general setting is given in~\cite{ABK24, AK24}. For the purposes of proving Theorem~\ref{thm:main},  we are able to get away with a simpler perspective based on  \emph{iterated homogenization}, available since the construction of the field $u$ relies solely on space-time periodic building blocks.

Fix $\kappa = \kappa_M \in \mathcal{K}$, and define $\theta^\kappa$ to be the solution the advection-diffusion equation~\eqref{eq:drift:diffusion} with drift velocity $u$ and smooth initial condition $\theta_{\mathsf{in}}$. Then, for each $m\in \{0,\ldots,M\}$, define $\theta_m$ to be the solution of the advection-diffusion equation with the same initial datum, but with \emph{drift velocity given by $u$ mollified at scale $\eps_m$} (which in our context means that the velocity is taken to equal $u_m$), and with diffusivity parameter $\kappa_m$ (as defined by \eqref{eq:kappa:m:real:1}--\eqref{eq:kappa:m:real:2}). That is, $\theta_m$ is defined as the solution of
\begin{equation}
\label{eq:theta:m:evo}
\partial_t \theta_m + u_m \cdot\nabla\theta_m - \kappa_m \Delta \theta_m =0 \,,
\qquad 
\theta_m|_{t=0} = \theta_{\mathsf{in}} 
\,,
\end{equation}
for each $m\in \{0,\ldots,M\}$.
With this notation, we say that 
\begin{center}
\emph{the operator $\partial_t + u\cdot \nabla - \kappa \Delta$ ``coarse-grains'' at scale $\eps_m$ to the operator $\partial_t + u_m \cdot \nabla - \kappa_m \Delta$ },
\end{center} 
if we are able to show that we have
\begin{equation}
\label{eq:coarse:graining}
\tfrac{1}{C_*} \cdot \kappa_m \|\nabla \theta_m\|_{L^2_{t,x}}^2 \leq 
\kappa \|\nabla \theta^\kappa\|_{L^2_{t,x}}^2 
\leq C_* \cdot \kappa_m \|\nabla \theta_m\|_{L^2_{t,x}}^2 \,,
\end{equation}
for some universal constant $C_* > 1$ (independent of $m$). 
If we are able to prove \eqref{eq:coarse:graining},  then since $u_0 = 0$ and $\kappa_0 \simeq 1$, by the energy balance for \eqref{eq:theta:m:evo} with $m=0$, and appealing to the Poincar\'e inequality,  we have $\kappa_0 \|\nabla \theta_0\|_{L^2_{t,x}}^2 \simeq \varrho \|\theta_{\mathsf{in}}\|_{L^2}^2$ for some $\varrho \in (0,\frac 12]$; this would then prove  Theorem~\ref{thm:main}. The proof of~\eqref{eq:coarse:graining} is described next.

\subsection{The heart of the matter: one step of the homogenization iteration} 
\label{sec:key}

In order to establish~\eqref{eq:coarse:graining} we take an iterative approach: using quantitative homogenization (akin to~\S\ref{sec:dispersion}) for each $m \in \{1,\ldots,M\}$ we show that the amount of energy dissipation experienced by $\theta_m$ is comparable to the amount of energy dissipation experienced by $\theta_{m-1}$. In this ``one step'' of the homogenization process we make an error, and we aim to prove that these errors are summable in $m$, leading to~\eqref{eq:coarse:graining}. 

Lemma~\ref{lem:hom:one:step} (below) summarizes this ``one step'' of homogenization induction, and this is the key step in the proof of Theorem~\ref{thm:main}. Its proof consists of  comparing the equations for~$\theta_m$ and~$\theta_{m-1}$ by homogenizing the oscillations at scale~$\eps_m$. The fact that we can only get away with worrying about one scale at each step of the proof is because the scales~$\eps_m$ are supergeometric, and thus becoming more well-separated as they get smaller:~\eqref{eq:eps:m:def} shows that ~$\eps_{m} / \eps_{m-1} \simeq \eps_{m-1}^{q-1} \to 0$ as $m\to \infty$. This is an undesirable and non-physical feature of our model, but is necessary for such a ``naive'' homogenization procedure to work, since we rely on the summability (in $m$) of the errors we make. In more realistic examples, one should expect to see true geometric scale separation, and this necessitates examining the interaction between \emph{many} scales at each step of a renormalization iteration, and dealing with errors which are always of order one. There has been some recent progress~\cite{AK24,ABK24} in developing a general \emph{quantitative theory of coarse-graining} for elliptic and parabolic operators, which is based on quantitative homogenization theory and is capable of analyzing such models---using a strategy which is broadly similar to our proof of Theorem~\ref{thm:main}. 

The precise statement of the main homogenization estimate in our induction is given in~\cite[Proposition 5.2]{AV24}; here we offer a simplified version (see also~Remark~\ref{rem:all:data}):
\begin{lemma}[The homogenization iteration]
\label{lem:hom:one:step} Fix $\alpha \in (0,\frac 13)$. There exists a small parameter $\delta = \delta(\alpha) \approx (q-1)^2 > 0$ and a universal constant $C= C(\alpha)>0$, such that if $\eps_1 = \Lambda^{-1}$ is sufficiently small (in terms of $\alpha$ alone) then the following statement holds. Assume the zero-mean initial datum $\theta_{\mathsf{in}}$ is real-analytic, with analyticity radius $R_{\theta_{\mathsf{in}}}\in (0,1]$ which satisfies $R_{\theta_{\mathsf{in}}} \geq \eps_1^{1+\frac{\gamma}{2}}$ (recall, $\gamma$ is defined in~\eqref{eq:kappa:m:explicit}).
For $\kappa \in \mathcal{K}$, define $M$ via~\eqref{eq:kappa:m:real:1}. 
Then, for all $m \in \{2,\ldots,M\}$ the solutions of~\eqref{eq:theta:m:evo} satisfy
\begin{equation}
\label{eq:big:lemma:1}
    \|\theta_m - \theta_{m-1}\|^2_{L^\infty([0,1];L^2(\TT^2))} +  \kappa_m \|\nabla \theta_m - \nabla \Tilde{\theta}_m\|^2_{L^2([0,1]\times\TT^2)} \leq C \eps^{2 \delta}_{m-1}   \|\theta_{\mathsf{in}}\|_{L^2(\TT^2)}^2
\end{equation}
and
\begin{equation}
\label{eq:big:lemma:2}
    \Bigg |\frac{\kappa_m \|\nabla \theta_m\|_{L^2_{t,x}}^2}{\kappa_{m-1}\|\nabla \theta_{m-1}\|_{L^2_{t,x}}^2} -1 \Bigg| \leq C \eps^\delta_{m-1}
    \,,
\end{equation}    
for a suitably defined function $\Tilde{\theta}_m$ (see~\eqref{e.actual.ansatz} below). Here $L^2_{t,x} = L^2([0,1]\times\TT^2)$.
\end{lemma}

\begin{remark}[All real-analytic data]
\label{rem:all:data}
We show in~\cite[Proposition 5.2]{AV24} that
Lemma~\ref{lem:hom:one:step} holds for all zero-mean real-analytic data $\theta_{\mathsf{in}}$,\footnote{Note that Theorem~\ref{thm:main} is stated for $H^1(\TT^d)$--smooth initial data $\theta_{\mathsf{in}}$, not real-analytic one. However, upon mollification any $H^1(\TT^d)$ data can be well-approximated by a real-analytic one, and this suffices for the proof to close; see \S\ref{sec:proof:conclude}.} no matter how small their analyticity radius\footnote{We say that $f \colon \TT^d \to \RR$    is real-analytic with analyticity radius (at least) $R_f \in (0,1]$ if $\sup_{n\geq 0} \| \nabla^n f\|_{L^2(\TT^d)} R_f^n (n!)^{-1} \leq \|f\|_{L^2(\TT^d)}$.}$R_{\theta_{\mathsf{in}}}$ is. The only difference is that the bounds in~\eqref{eq:big:lemma:1} and~\eqref{eq:big:lemma:2} hold only for $m \in \{ m_{*}+1, \ldots, M \}$, where $m_{*} := \min \{ m \in \mathbb{N} \colon m\geq 1, R_{\theta_{\mathsf{in}}} \geq \eps_{m}^{1 + \frac{\gamma}{2}} \}$. 
\end{remark}

\begin{proof}[Discussion of the proof of Lemma~\ref{lem:hom:one:step}]
The proof of this lemma is the bulk of the paper~\cite{AV24}. It requires solving separate space and a time homogenization problems, due to the three different ``fast'' scales in the equation for $\theta_m$:  the smallest spatial scale is $\eps_m$, which is the length scale of oscillation of the shear flows;
the smallest time scale, on which the shear flows switch directions,  is given by $\tau_m$; lastly, we have the time scale $\tau_{m-1}^{\prime\prime}$ on which the Lagrangian flow maps $X_{m-1,\ell}$ must refresh. 
We should think of these three fast scales as being well-separated, with the spatial scale $\eps_m$ being the
smallest/fastest, and with $\tau_m \ll \tau_{m-1}^{\prime\prime}$ (see~\eqref{eq:tau:m:cond:2}).

The idea is the same as for the classical periodic case presented in \S\ref{sec:two-scale} and~\S\ref{sec:dispersion}. We want to make a two-scale ansatz $\tilde{\theta}_m$, similar to the one for~$\tilde{\theta}_\eps$ in~\eqref{e.twoscale.wep.para}. Here, the ``macroscopic'' function is~$\theta_{m-1}$, and akin to~\eqref{e.advec.s.approx}--\eqref{e.advec.s.f} we want to plug this two-scale  ansatz into the ``microscopic'' equation, which is the equation for~$\theta_m$ (see~\eqref{eq:theta:m:evo}). The hope is that the resulting error is small in $L^2_t H^{-1}_x$, so that we can obtain good estimates by appealing to parabolic regularity theory, as in the proof of Lemma~\ref{lem:Taylor:quantitative}.

The initial ``naive'' idea would be to define $\Tilde{\theta}_m$ by the two-scale ansatz 
\begin{equation*}
\Tilde{\theta}_m = \theta_{m-1} + 
\sum_{k,\ell \in \ZZ} 
\xi\left(\tfrac{\cdot}{\tau_{m}}-k \right) 
\hat{\xi}\left(\tfrac{\cdot}{\tau_{m-1}^{\prime\prime}}-\ell \right)   
\sum_{i \in \{1,2\}}
(\chi_{m,k,e_i} \circ X^{-1}_{m-1,\ell}) \partial_i \theta_{m-1}  + H_m
\,,
\end{equation*}
where the $\chi_{m,k,e_i}$ are fast periodic ``correctors'' with slope $e_i$,  which are  oscillating at scale $\eps_m$ and are obtained by solving a cell-problem related to the shear flow $\Psi_k$; the $H_m$ is responsible for homogenizing the time cutoffs, and the $\xi$, $\hat{\xi}$ are unit-scale time cutoffs that are slightly wider than $\zeta$ and $\hat{\zeta}$.  
Unfortunately, compared to the periodic setting, we face a number of additional complications and for these reasons, the definition above is not precise enough and cannot work. First, our equation is not $\eps_m$--periodic because of ``macroscopic'' perturbations from the scales~$\eps_{m-1}, \ldots,\eps_0$, including those caused by the Lagrangian flows. Moreover, we cannot just ``freeze'' the coefficients and make an error here; some of these macroscopic perturbations need to be homogenized along with the~$\eps_m$--scale oscillations. For these reasons, our two-scale ansatz is a bit more complicated. 

\smallskip

The \emph{actual} definition of~$\Tilde{\theta}_m$ which is used in~\eqref{eq:big:lemma:1} is given in~\cite[(4.24)]{AV24} as
\begin{align}
\label{e.actual.ansatz}
\Tilde{\theta}_m
&=
T_{m-1} 
+
\sum_{k,\ell \in \ZZ} 
\xi\left(\tfrac{\cdot}{\tau_{m}}-k \right) 
\hat{\xi}\left(\tfrac{\cdot}{\tau_{m-1}^{\prime\prime}}-\ell \right)   
\sum_{i\in\{1,2\}}
\tilde{\chi}_{m,k,e_i}
\bigl(
\nabla \bigl( T_{m-1} \circ X_{m-1,\ell} \bigr) \circ X_{m-1,\ell}^{-1} \bigr)_i
+
\widetilde{H}_{m}
\,,
\end{align}
where:
\begin{itemize}[leftmargin=*]
\item The ``twisted correctors''~$\tilde{\chi}_{m,k,e_i}$ play the role of the fast periodic correctors with slope $e_i$. These  are simply the periodic correctors for the horizontal and vertical shears at scale $\eps_m$, denoted by $\chi_{m,k,e_i}$ (see~\cite[(3.6)]{AV24}), evaluated along the inverse flow $X_{m-1,k}^{-1}$  (see~\cite[(4.26)]{AV24}). Note that in~\eqref{e.actual.ansatz} the twisted correctors $\tilde{\chi}_{m,k,e_i}$, which are indexed by $k \in \ZZ$, are simply glued together in time by the partition of unity given by  $\xi(\frac{\cdot}{\tau_{m}}-k )$; all other terms in~\eqref{e.actual.ansatz} are indexed by $\ell \in \ZZ$. 
\item The $\tilde{H}_m$ is an approximate time corrector, see the below discussion and~\cite[(4.21)]{AV24}.
\item The function~$T_{m-1}$ (defined in~\cite[(4.13)]{AV24}) is a perturbation of the function~$\theta_{m-1}$. It is chosen to be an \emph{approximate} solution of the equation
\begin{equation}
\label{e.Tm.divform}
\partial_t T_{m-1} 
+ u_{m-1} \cdot \nabla T_{m-1}
- \nabla \cdot \bigl( 
\mathbf{K}_m
+ \mathbf{s}_{m-1} \bigr) \nabla T_{m-1} 
=0
\,, 
\qquad 
T_{m-1}|_{t=0} = \theta_{\mathsf{in}}
\,,
\end{equation}
where
\begin{itemize} [leftmargin=*]
\item The field~$\mathbf{K}_m(t) $ (see~\cite[(3.26)]{AV24}) is a time-oscillating  field that oscillates between the homogenized matrices for the vertical and horizontal shears; it averages to~$\kappa_{m-1}\Itwo$, plus a negligible error.

\item The field~$\mathbf{s}_{m-1}$ is defined as 
\begin{equation*}
\mathbf{s}_{m-1}:= 
\mathbf{K}_m
\sum_{\ell\in\ZZ}  
\hat{\xi}\left(\tfrac{\cdot}{\tau_{m-1}^{\prime\prime}}-\ell \right)   
\bigl( \nabla X_{m-1,\ell} \circ X_{m-1,\ell}^{-1} - \Itwo\bigr)
\,.
\end{equation*}
It is added to the diffusion matrix of the $T_{m-1}$ equation in order to deal with ``distortions'' caused by the Lagrangian flows; these distortions are small, but cannot be neglected in the equation for~$\theta_m$. 
\end{itemize}
Equation~\eqref{e.Tm.divform} is close to the equation for~$\theta_{m-1}$ (see~\eqref{eq:theta:m:evo} with $m$ replaced by $m-1$) in the sense that~$\mathbf{s}_{m-1} $ is small (due to~\eqref{eq:tau:m:cond:1}) and the time oscillating field~$\mathbf{K}_m(t) $ will homogenize to $\kappa_{m-1} \Itwo$, since time oscillations simply average.
\end{itemize}

\emph{A warning:} for technical reasons, we do \emph{not} define the function~$T_{m-1}$ to be a solution of equation~\eqref{e.Tm.divform}! The reason is that equation~\eqref{e.Tm.divform} does not possess sufficiently strong regularity estimates; this is because when compared to~$\kappa_{m-1}\Itwo$, the matrix~$\mathbf{K}_m$ has ellipticity constants which are much worse. Our regularity estimates for~$T_{m-1}$ must be very sharp, since it is the ``macroscopic'' function here and, as we saw in \S\ref{sec:two-scale} and \S\ref{sec:dispersion}, higher derivatives of the macroscopic function show up in the quantitative homogenization estimates. 
To fix this problem, we choose~$T_{m-1}$ to be an approximate solution of~\eqref{e.Tm.divform} which  we construct by hand, using a high-order iterative scheme. We cut off the iteration defining $T_{m-1}$ at a finite stage, ensuring that the regularity estimates we need to hold are valid, while making a very small error in the equation for~$T_{m-1}$.  It is in these high-order estimates for $T_{m-1}$ that we use the quantitative real-analyticity assumption of the initial data  present in the statement of Lemma~\ref{lem:hom:one:step} (we refer the reader to~\cite[Section 4.3]{AV24}).

Similarly, the function~$\tilde{H}_m$ in~\eqref{e.actual.ansatz} plays the role of the time corrector (analogous to~$h_e$ in~\S\ref{sec:two-scale}). However, as for~$T_{m-1}$, in order to ensure that we have good enough regularity estimates for~$\tilde{H}_m$, we actually construct it as an approximate time corrector using an iterative scheme, to ensure that it has better regularity properties. 

With the definition of the ansatz~\eqref{e.actual.ansatz} in hand, the argument in~\cite{AV24} proceeds by plugging~$\tilde{\theta}_{m}$ into the equation for~$\theta_m$, and checking that the right-hand side is small. The error is quite explicit, although we do not write it here since it takes half a page to display the nine error terms: see~\cite[(5.18)--(5.26)]{AV24}. It is then a tedious task to check that all these terms are suitably small in~$L^2_tH^{-1}_x$, but after doing so we obtain the bounds~\eqref{eq:big:lemma:1}--\eqref{eq:big:lemma:2} claimed in the lemma. 
\end{proof}

\subsection{The proof of Theorem~\ref{thm:main}: anomalous diffusion} 
\label{sec:proof:conclude}

To conclude this section, we show how Lemma~\ref{lem:hom:one:step} (and Remark~\ref{rem:all:data}) imply that the family of solutions $\{\theta^\kappa\}_{\kappa \in \mathcal{K}}$ of the advection-diffusion equation~\eqref{eq:drift:diffusion} with velocity $u$, exhibits anomalous diffusion in the sense of Definition~\ref{def:ad}, as promised in the statement of Theorem~\ref{thm:main}. For the proofs of the statement made in bullets three--six in Theorem~\ref{thm:main}, we refer the reader to~\cite{AV24}.

Assume for the moment that the function $\theta_{\mathsf{in}}$ is real-analytic.\footnote{At the end of the proof we will use the fact that $H^1(\TT^2)$ functions can be approximated by real-analytic ones, with an error that can be made quantitatively small in $L^2(\TT^2)$.} Denote by $R_{\theta_{\mathsf{in}}} \in (0,1]$ its analyticity radius, so that 
$\sup_{n\geq 0} \| \nabla^n \theta_{\mathsf{in}}\|_{L^2(\TT^d)} R_{\theta_{\mathsf{in}}}^n (n!)^{-1} \leq \|\theta_{\mathsf{in}}\|_{L^2(\TT^d)}$. Define $m_{*} \geq 1 $ such that $\eps^{1+\frac{\gamma}{2}}_{m_{*}} \simeq R_{\theta_{\mathsf{in}}}$; the precise definition of $m_{*}$ is given in Remark~\ref{rem:all:data}.
This definition, together with the relation $\kappa_m \simeq \eps_m^{1+\alpha+\gamma}$ for all $m\geq 0$ (see~\eqref{eq:kappa:m:explicit}), gives
\begin{equation}
\label{eq:kappa:m:in}
\kappa_{m_{*}} 
\simeq 
\eps_{m_{*}}^{1+\alpha+\gamma}
\simeq 
R_{\theta_{\mathsf{in}}}^{\frac{2(1+\alpha+\gamma)}{2+\gamma}}
\,.
\end{equation}
Relation~\eqref{eq:kappa:m:in} shows that $\kappa_{m_{*}}$  is a function of $R_{\theta_{\mathsf{in}}}$ alone (it is independent of the molecular diffusivity parameter $\kappa$). By the energy inequality for the $\theta_{m_{*}}$ evolution (that is,~\eqref{eq:theta:m:evo} with $m$ replaced by $m_{*}$), and the Poincar\'e inequality, we deduce
\begin{equation*}
\tfrac 12 \tfrac{d}{dt} \| \theta_{m_{*}} (t,\cdot)\|_{L^2(\TT^2)}^2 
= 
- \kappa_{m_{*}} \|\nabla \theta_{m_{*}} (t,\cdot)\|_{L^2(\TT^2)}^2 
\leq 
-  (2\pi)^2 \kappa_{m_{*}} \| \theta_{m_{*}} (t,\cdot)\|_{L^2(\TT^2)}^2 
\,.
\end{equation*}
In turn, the above estimate implies 
\begin{align}
\label{eq:kappa:m:*:diffusion}
 \kappa_{m_{*}} \|\nabla \theta_{m_{*}}\|_{L^2([0,1]\times \TT^2)}^2 
 &= \tfrac 12 \|\theta_{\mathsf{in}}\|_{L^2(\TT^2)}^2 
 - \tfrac 12 \|\theta_{m_*}(1,\cdot)\|_{L^2(\TT^2)}^2 
 \geq  \tfrac 12 \|\theta_{\mathsf{in}}\|_{L^2(\TT^2)}^2 \bigl( 1- e^{-8\pi^2 \kappa_{m_*}} \bigr)
 \geq 2 \kappa_{m_*} \cdot \tfrac 12 \|\theta_{\mathsf{in}}\|_{L^2(\TT^2)}^2 
 \,.
\end{align}
Here we have used that without loss of generality $\kappa_{m_*} \leq \frac 13$ (since $m_* \geq 1$, and we can take $\Lambda$ to be sufficiently large), and so $1- e^{-8\pi^2 \kappa_{m_*}} \geq 2 \kappa_{m_*}$.

Next, take any $\kappa \in \mathcal{K}$, and define $M = M(\kappa) \geq 0$ according to \eqref{eq:kappa:m:real:1}, so that $\kappa = \kappa_M \simeq \eps_M^{1+\alpha+\gamma}$; in particular, as $\kappa \to 0$ we have $M\to\infty$, and so $M > m_*$ for all sufficiently small $\kappa$ (since $m_*$ is only a function of $R_{\theta_{\mathsf{in}}}$, which is fixed). By  Remark~\ref{rem:all:data}, the bounds \eqref{eq:big:lemma:1}--\eqref{eq:big:lemma:2} hold for all $m$ in the nontrivial range $\{m_* + 1,\ldots,M\}$. For instance, by telescoping  estimate \eqref{eq:big:lemma:2}, we obtain that  for all $m\in \{m_*,\ldots,M\}$ 
\begin{equation*}
\prod_{m = m_*+1}^M (1 - C \eps_{m-1}^\delta) 
\leq 
\frac{\kappa_m \|\nabla \theta_m\|_{L^2([0,1]\times \TT^2)}^2}{\kappa_{m_{*}} \|\nabla \theta_{m_{*}}\|_{L^2([0,1]\times \TT^2)}^2 } 
\leq 
\prod_{m = m_*+1}^M (1 + C \eps_{m-1}^\delta) \,.
\end{equation*}
The super-geometric growth  $\eps_m \simeq \eps_{m-1}^q$ with $q>1$, together with the fact that $\eps_1 = \Lambda^{-1}$ is chosen sufficiently small (in terms of $\alpha$), implies  
\begin{equation*}
    \frac{3}{4} \leq \prod^\infty_{m = 2} (1 \pm C\eps^\delta_{m-1}) \leq \frac{4}{3}
    \,,
\end{equation*}
and thus for all $m\in \{m_*,\ldots,M\}$, we have
\begin{equation}
\label{eq:main:thm:est2}
\tfrac{3}{4}\kappa_{m_*} \|\nabla \theta_{m_*}\|_{L^2([0,1]\times \TT^2)}^2 
\leq 
\kappa_M \|\nabla \theta_{M}\|_{L^2([0,1]\times \TT^2)}^2 
\leq 
\tfrac{4}{3} \|\nabla \theta_{m_*}\|_{L^2([0,1]\times \TT^2)}^2
\,.
\end{equation}
Estimate~\eqref{eq:main:thm:est2} is used with $M = m_*$ in the proof.

In view of the previously established lower bound~\eqref{eq:kappa:m:*:diffusion}, it remains to relate the term $\kappa_M \|\nabla \theta_{M}\|_{L^2([0,1]\times \TT^2)}^2 = \kappa \|\nabla \theta_{M(\kappa)}\|_{L^2([0,1]\times \TT^2)}^2$ appearing in~\eqref{eq:main:thm:est2}, to the true energy dissipation rate appearing in~\eqref{eq:anomaly}, namely $\kappa \|\nabla \theta^\kappa \|_{L^2([0,1]\times \TT^2)}^2$. In order to achieve this, we note that $\theta^\kappa - \theta_{M}$ satisfies the equation
\begin{equation*}
(\partial_t - \kappa \Delta + u \cdot \nabla)(\theta^\kappa - \theta_M)
= \ddiv \bigl( (\phi - \phi_M) \nabla \theta_M \bigr) 
\,,
\qquad 
(\theta^\kappa - \theta_M)|_{t=0}=0
\,.
\end{equation*}
Since by construction we have $\|\phi - \phi_{M}\|_{L^\infty_{x,t}} \leq C\eps^{(1+\alpha)q}_{M} \ll 1$ (which follows by telescoping the trivial bound for~\eqref{eq:phi:increment:good}), a simple energy estimate gives
\begin{equation}\label{eq:main:thm:est3}
\|\theta^\kappa - \theta_{M}\|_{L^\infty([0,1];L^2(\TT^2))}^2 
+ 
\kappa \|\nabla \theta^\kappa - \nabla \theta_{M}\|_{L^2([0,1]\times\TT^2)}^2 
\leq 
\tfrac{1}{8} \kappa \|\nabla \theta_{M}\|_{L^2([0,1]\times\TT^2)}^2 
    \,.
\end{equation}
Combining \eqref{eq:kappa:m:in},  \eqref{eq:kappa:m:*:diffusion}, \eqref{eq:main:thm:est2}, \eqref{eq:main:thm:est3}, and recalling that $\kappa = \kappa_M$, we deduce
\begin{align}
\label{eq:proof:almost:over}
\kappa \|\nabla \theta^\kappa \|_{L^2([0,1]\times\TT^2)}^2 
&\geq \kappa \|\nabla \theta_{M}\|_{L^2([0,1]\times\TT^2)}^2  
- \kappa \|\nabla \theta^\kappa - \nabla \theta_{M}\|_{L^2([0,1]\times\TT^2)}^2 
\notag \\
&\geq \tfrac{7}{8} \kappa_M \|\nabla \theta_{M}\|_{L^2([0,1]\times\TT^2)}^2 
\geq \tfrac{7}{8} \cdot \tfrac{3}{4} \cdot \kappa_{m_*} \|\nabla \theta_{m_*}\|_{L^2([0,1]\times\TT^2)}^2  
\notag \\
&\geq \kappa_{m_*} \cdot \tfrac{1}{2}  \|\theta_{\mathsf{in}}\|_{L^2(\TT^2)}^2 
\geq C_* R_{\theta_{\mathsf{in}}}^{\frac{2(1+\alpha+\gamma)}{2+\gamma}} \cdot \tfrac{1}{2}  \|\theta_{\mathsf{in}}\|_{L^2(\TT^2)}^2 
    \,.
\end{align}
In order to complete the proof of anomalous diffusion for all zero-mean data in $H^1(\TT^2)$, we need to address the real-analyticity assumption used in~\eqref{eq:proof:almost:over}, which explicitly contains the analyticity radius $R_{\theta_{\mathsf{in}}}$.

For this purpose, for any $\theta_{\mathsf{in}} \in H^1(\TT^2)$  we define $L_{\mathsf{in}}:= \frac{\|\theta_{\mathsf{in}}\|_{L^2}}{\|\nabla \theta_{\mathsf{in}}\|_{L^2}}$. Then, we mollify $\theta_{\mathsf{in}} $ at spatial scale $\lambda L_{\mathsf{in}}$, for a parameter $\lambda$ which is to be determined; that is, we define
\begin{equation*}
\tilde{\theta}_{\mathsf{in}}
:=
\theta_{\mathsf{in}} \ast \Phi(\lambda^2 L_{\mathsf{in}}^2,\cdot)
\,,
\end{equation*}
where $\Phi(t,x)$ is the standard heat kernel (with unit diffusion coefficient). The function $\tilde{\theta}_{\mathsf{in}}$ is real-analytic, and its analyticity radius is given by $R_{\tilde{\theta}_{\mathsf{in}}} = c \lambda L_{\mathsf{in}}$, for a universal constant $c>0$. Moreover, standard properties of the heat kernel  $\Phi$ imply that 
\begin{align}
\label{eq:analytic:approximation}
\| \theta_{\mathsf{in}}\|_{L^2(\TT^2)}^2 
\geq
 \| \tilde{\theta}_{\mathsf{in}}\|_{L^2(\TT^2)}^2 
&\geq 
\| \theta_{\mathsf{in}}\|_{L^2(\TT^2)}^2 
- C \lambda^2 L_{\mathsf{in}}^2 \|\nabla \theta_{\mathsf{in}}\|_{L^2(\TT^2)}^2
\geq 
 \bigl( 1 - C \lambda^2\bigr) \| \theta_{\mathsf{in}}\|_{L^2(\TT^2)}^2 
 \,.
\end{align}
To conclude, we denote by $\tilde{\theta}^\kappa$ the solution of the advection-diffusion equation~\eqref{eq:drift:diffusion} with initial condition $\tilde{\theta}_{\mathsf{in}}$. By linearity, we have that $(\partial_t + u\cdot \nabla -\kappa \Delta)(\tilde{\theta}^\kappa - \theta^\kappa) = 0$, and so by the standard energy balance for this difference (recall~\eqref{eq:theta:balance}), the triangle inequality, the bound~\eqref{eq:proof:almost:over} with $\theta^\kappa$ replaced by $\tilde{\theta}^\kappa$ and $\theta_{\mathsf{in}}$ replaced by $\tilde{\theta}_{\mathsf{in}}$, and with the bounds in \eqref{eq:analytic:approximation},  we deduce 
\begin{align}
\kappa \| \nabla \theta^\kappa\|_{L^2([0,1]\times\TT^2)}^2
&\geq 
\kappa \| \nabla \tilde{\theta}^\kappa\|_{L^2([0,1]\times\TT^2)}^2
-
\kappa \| \nabla (\tilde{\theta}^\kappa -\theta^\kappa) \|_{L^2([0,1]\times\TT^2)}^2
\notag\\
&\geq C_* R_{\tilde{\theta}_{\mathsf{in}}}^{\frac{2(1+\alpha+\gamma)}{2+\gamma}} \cdot \tfrac{1}{2}  \|\tilde{\theta}_{\mathsf{in}}\|_{L^2(\TT^2)}^2 
- \tfrac 12 \|\tilde{\theta}_{\mathsf{in}} -\theta_{\mathsf{in}} \|_{L^2(\TT^2)}^2 
\notag\\
&\geq \tfrac{C_*}{2} (c \lambda L_{\mathsf{in}})^{\frac{2(1+\alpha+\gamma)}{2+\gamma}} \cdot \tfrac{1}{2}  \|\theta_{\mathsf{in}}\|_{L^2(\TT^2)}^2 
- C \lambda^2 \cdot \tfrac 12 \|\theta_{\mathsf{in}} \|_{L^2(\TT^2)}^2 \,,
\label{eq:finally}
\end{align}
assuming that $C \lambda^2 \leq \frac 12$. The bound~\eqref{eq:finally} dictates that we choose $\lambda$ small enough ensure that the first term on the right side dominates; this amounts to letting 
\begin{equation*}
\lambda \leq c_* L_{ \mathsf{in}}^{ \frac{1+\alpha+\gamma}{1-\alpha}} 
\,,
\end{equation*}
for some universal constant $0 < c_* \ll 1$ which depends on $c, C$, and $C_*$. Inserting the above bound into~\eqref{eq:finally} implies 
\begin{equation*}
\kappa \| \nabla \theta^\kappa\|_{L^2([0,1]\times\TT^2)}^2
\geq  \underbrace{\Bigl( \tfrac{C_*}{4} (c c_* )^{\frac{2(1+\alpha+\gamma)}{2+\gamma}}\Bigr) \cdot L_{\mathsf{in}}^{\frac{2(1+\alpha+\gamma)}{1-\alpha}}}_{=: \varrho} \cdot \tfrac{1}{2}  \|\theta_{\mathsf{in}}\|_{L^2(\TT^2)}^2 
\,.
\end{equation*}
The dependence of the coefficient $\varrho$ on $L_{\mathsf{in}}$ claimed in the first bullet of Theorem~\ref{thm:main} now follows by letting $\eps = \frac{2\gamma}{1-\alpha} = \frac{2(q-1)(1+\alpha)}{(q+1)(1-\alpha)} \to 0 $ as $q\to 1$.

\section{Extensions and open problems}
\label{sec:open}

\subsection{Uniform H\"older regularity of the scalars}
\label{sec:uniformholder}

An important question which is not addressed by Theorem~\ref{thm:main} is the \emph{regularity} of the solutions~$\theta^\kappa$, in particular whether there is a uniform-in-$\kappa$ H\"older estimate such as
\begin{equation}
\label{e.Holder.beta}
\sup_{\kappa \in \mathcal{K}} 
\|  \theta^\kappa \|_{L^2([0,1]; C^{0,\bar{\alpha}}(\TT^d))}
< \infty\,.
\end{equation}
If we could show~\eqref{e.Holder.beta} for~$\bar{\alpha}$ arbitrarily close to~$\frac{1-\alpha}{2}$, then this would be resolve part (ii) of the Obukhov-Corrsin dichotomy discussed in \S\ref{sec:OC}.

To see why we should expect such an estimate, let us recall the Campanato characterization of the H\"older space~$C^{0,\beta}$, which states that\footnote{We denote by $\{\eta_r\}_{r>0}$ a standard family of mollifiers, with $\eta_r$ localized to scale $\simeq r$, and by~$\{ \eta_{r,s}\}_{r,s>0}$ a family space-time mollifiers at spatial scale~$r$ and time scale~$s$.} for any~$f \in H^1(\TT^d)$,  
\begin{equation}
\label{e.Campanato}
\bigl[ f \bigr]_{C^{0,\beta}(\TT^d)} 
\simeq 
\sup_{r\in (0,1)}  
r^{1-\beta} 
\| \nabla (f\ast\eta_r) \|_{L^\infty(\TT^d)} 
\,.
\end{equation}
What we have shown in Theorem~\ref{thm:main} (recall~\eqref{eq:main:thm:est2} and also~\eqref{eq:kappa:m:explicit}) is that 
\begin{equation*}
\| \nabla \theta_m \|_{{L}^2([0,1]\times\TT^d)} 
\approx \kappa_m^{-\frac12}
\simeq \eps_m^{-\frac{1+\alpha+\gamma}{2}},
\end{equation*}
where we ignore the dependence on the initial datum $\theta_{\mathsf{in}}$ of the implicit constants in the $\approx$ symbol, and we recall that~$\gamma>0$ is a very small exponent (proportional to~$q-1$, which can thus be made arbitrarily small in our construction of the vector field). 
While not explicitly stated, the estimate in Lemma~\ref{lem:hom:one:step} implies that, for~$j<m$, we have 
\begin{equation}
\label{e.conv}
\kappa_j
\| \nabla \theta_m \ast \eta_{\eps_j',\tau_j'''} - \nabla \theta_j \|^2_{{L}^2([0,1]\times\TT^d)} 
\les \eps_{j-1}^{2\delta} \,,
\end{equation}
where~$\eps_j'= \sqrt{ \eps_{j+1}\eps_{j}}$ and~$\tau_j''':= \sqrt{ \tau_{j} \tau_{j}''}$ are chosen so that~$\eps_j < \eps_{j'} < \eps_{j-1}$ and~$\tau_j'' < \tau_j''' < \tau_j$, with some amount of separation between each of these scales; again, we ignore the dependence on the initial datum $\theta_{\mathsf{in}}$ of the implicit constants in the $\les$ symbol.
In other words, we should consider~$\nabla \theta_j \simeq  \nabla \theta_m \ast  \eta_{\eps_j',\tau_j'''}$ if~$j<m$. Therefore the previous two displays imply the bound 
\begin{equation*}
\| \nabla (\theta_m \ast  \eta_{\eps_j',\tau_j'''}) \|_{{L}^2([0,1] \times \TT^d)} 
\approx \kappa_j^{-\frac12}
\simeq \eps_j^{-\frac{1+\alpha+\gamma}{2}}
\simeq 
(\eps_j')^{-\frac{1+\alpha}{2} -O(\gamma)}
\end{equation*}
\emph{If we could upgrade this bound from~$L^2$ to~$L^\infty$},\footnote{We would also need some uniform-in-time estimates to treat the time mollification.} then by~\eqref{e.Campanato} we would have the desired uniform H\"older estimate for~$\theta_m$ with regularity exponent~$\bar{\alpha} = 1 - \frac{1+\alpha}{2} - O(\gamma) = \frac{1-\alpha}{2} - O(\gamma)$. 

The aforementioned ``upgrade'' amounts to proving the Lipschitz-type estimate \begin{equation}
\label{e.Lipschitz}
\| \nabla \theta_m \|_{L^2_tL^\infty_x ([0,1] \times \TT^d)} 
\approx \kappa_m^{-\frac12}
\,,
\end{equation}
which states that the gradient field~$ \nabla \theta_m$ does not concentrate on sets of small measure: it is roughly the same size, everywhere in~$\TT^d$.  It is very natural to expect such an estimate to be true. First, at an intuitive level, since: (i) our vector field is built from periodic ingredients with supergeometric scale separation between successive scales (so that the main contribution of~$\nabla \theta_m$ are the wiggles at scale~$\eps_m$), and (ii) we have shown that the solutions have expansions with periodic ingredients, there is no reason to expect any such ``concentration'' to occur. 

Second, uniform Lipschitz estimates of exactly this type have played a central role in homogenization theory since the pioneering work of Avellaneda and Lin~\cite{AL1,AL2} in the~'80s; they now go by the name \emph{large-scale regularity theory} (see the monographs~\cite{AKMBook,ShenBook} and the references therein). 
While the arguments of Avellaneda and Lin~\cite{AL1,AL2} are based on compactness and apply only to equations with periodic coefficients and thus one active scale, a quantitative approach to large-scale regularity was later proposed in~\cite{AS16}; this method is more robust and applies to equations with non-periodic coefficients (such as almost periodic and random coefficients). 

The way it works heuristically is that the regularity of the homogenized equation is ``transferred'' to the equation with oscillating coefficients using quantitative homogenization estimates, and an excess decay iteration. In our setting, the ``homogenized equation'' would be the equation for~$\theta_{m-1}$ and the ``equation with oscillating coefficients'' would be the equation for~$\theta_m$. This sets us an induction \emph{going down the scales} in which the regularity for~$\theta_{m-1}$ is transferred to~$\theta_m$: a cascade of regularity, dual to the inverse cascade of homogenization. Such an iteration argument would obviously be very technical to implement and take a great deal of effort to write, but we have all the ingredients to implement it in our setting,\footnote{We would need a localized version of Lemma~\ref{lem:hom:one:step}, which applies not just in the whole torus~$\TT^2$, but in all appropriately scaled space-time cylinders above the scale at which we expect homogenization to occur. Such an estimate can be obtained by an argument similar to that of Lemma~\ref{lem:hom:one:step}, there will just be additional boundary layer terms to control.}  and we do not expect that formalizing it would require major new ideas beyond~\cite{AV24,AS16}. 

As part of the induction argument that proves~\eqref{e.Lipschitz}, we will also upgrade the estimate~\eqref{e.conv} to an~$L^\infty$-type estimate, and thus the uniform-in-$m$ H\"older estimate for~$\theta_m$ can be obtained. 

One nice consequence of such a H\"older estimate would be that the condition in Theorem~\ref{thm:main} that the initial data belong to~$H^1(\TT^d)$ would be removed, and the positive constant~$\varrho$ would be universal. This would then imply the exponential decay in time of the~$L^\infty(\TT^d)$-norm of our solutions~$\theta^{\kappa_j}$, uniformly along our subsequence~$\kappa_j$. 

\subsection{Geometric separation of scales and intermittency}

A major shortcoming of the vector field constructed in Section~\ref{sec:construct} is that the active scales~$\{ \eps_m\}_{m\geq 0}$ are supergeometrically separated. This nonphysical aspect was necessary for our proof of anomalous diffusion, based on iterative quantitative homogenization, to work. If we want the methods in~\cite{AV24} to apply to more realistic physical models, we need to develop a more flexible method.

Therefore, we think that it is an important open problem is to prove anomalous diffusion for a variant of our construction in which \emph{the scales~$\{ \eps_m \}_{m\geq 0}$ are geometrically separated},~i.e.,~$\eps_m = \exp(-Cm)$.

If the scales are geometrically separated, then we would make \emph{an error of order one at each scale} $m$---which we obviously cannot sum up over $m$! Moreover, we would have to deal with leading-order interactions between multiple scales at once. In contrast, the proof of Theorem~\ref{thm:main} is based on comparing the equation for~$\theta_{m-1}$ to the equation for~$\theta_{m}$; evidently the scales~$\eps_{m+1}$ and~$\eps_{m-1}$ do not interact directly, which would not be the case if they were geometrically separated. 

It is only due to the supergeometric scale separation that we should expect the H\"older estimate discussed in Section~\ref{sec:uniformholder} above to be valid. A geometric separation of scales---which implies the interaction between many scales at once---should lead to concentration effects and therefore intermittency. We would expect in this case that the~$L^p$ norm of~$\nabla \theta_m$ to be very different for different values of~$p$.

At first glance, these issues may seem fatal to our whole strategy based on homogenization theory. However, a general ``coarse-graining'' theory, based on quantitative homogenization methods, has been recently developed~\cite{AK22,AK24}. 
Based on this, and on an iterative quantitative homogenization argument, a superdiffusive central limit theorem was recently proved in~\cite{ABK24} for a passive scalar model with a random vector field (with non-separated scales). There is therefore some hope that iterative quantitative homogenization methods as in~\cite{AV24, ABK24} may apply to  more physically realistic models of scalar turbulence.

\subsection{A regularity threshold for Euler solutions exhibiting anomalous diffusion?}

We conjecture that the regularity threshold $\alpha = \frac 13$ present in Theorem~\ref{thm:main} is not an artifact of our proof in~\cite{AV24}, but rather, a delicate rigidity constraint imposed by the Lagrangian nature of turbulent diffusion. We propose the following problem:
\begin{conjecture}[Anomalous diffusion $+$ self advection $=$ Onsager super-criticality]
\label{eq:big:conj}
Let $d\in \{2,3\}$. \textbf{Assume} that $u \in C^0([0,1];C^\alpha(\TT^d)) \cap C^\alpha([0,1];C^0(\TT^d))$ is a weak solution of the incompressible Euler equations, for some H\"older regularity exponent $\alpha \in (0,1)$. Furthermore, \textbf{assume} that for any initial condition $\theta_{\mathsf{in}} \in \dot{H}^1(\TT^d)$, the family of solutions $\{\theta^\kappa\}_{\kappa>0}$ of the passive scalar equation~\eqref{eq:drift:diffusion} with velocity field $u$ displays anomalous diffusion continuously in time; i.e.,~\eqref{eq:anomaly} holds and the time dissipation measure $\mathcal{E}(dt)$ from~\eqref{eq:dissip:measure} is non-atomic. \textbf{Then, we have that $\alpha \leq \frac 13$.}
\end{conjecture}
We expect that the method of proof for resolving Conjecture~\ref{eq:big:conj} is  more interesting than what the statement of the Conjecture yields. Indeed, if one wishes to prove that the Conjecture is \emph{false}, then one first needs to construct $C^\alpha_{x,t}$ weak solutions of the incompressible Euler equations, for some $\alpha \in (\frac 13, 1)$. This question is widely open (in the absence of artificial forcing terms). Conversely, if one is to prove that the Conjecture is \emph{true}, then one needs to synthesize weak Euler solutions from the information gained by watching what the associated advection-diffusion equation does to  arbitrary~$H^1(\TT^d)$ smooth scalar initial conditions, in the vanishing diffusivity limit; one of the clear enemies is the pressure term in the Euler system.

\section*{Acknowledgements}
S.A.~was supported by NSF grant DMS-2350340. V.V.~was in part supported by the
Collaborative NSF grant DMS-2307681 and a Simons Investigator Award.
 

\end{document}